\documentclass{article}

\usepackage[T1]{fontenc}
\usepackage[utf8]{inputenc}
\usepackage{lmodern}
\usepackage{microtype}

\usepackage[total={5.8in, 7.8in}, asymmetric, bindingoffset=0.4in]{geometry}
\usepackage{amsmath}
\usepackage{amssymb}
\usepackage{amsthm}
\usepackage{mathtools}
\usepackage{enumerate}
\usepackage[dvipsnames]{xcolor}

\usepackage{multirow}
\usepackage{empheq}
\usepackage{listings}
\usepackage{color}

\usepackage[square,comma,numbers]{natbib}
\usepackage{hyperref}
\usepackage{graphicx} 
\usepackage[font=small,labelfont=bf]{caption}
\usepackage{tikz}
\usetikzlibrary{babel}
\usepackage{subfig}

\numberwithin{equation}{section}
\usepackage{pgfplots}
\pgfplotsset{width=10cm,compat=1.9}
\def\ae{\alpha^\eps}
\def\ap{\alpha^\prime}
\def\jg{J_\gamma}
\def\wg{w^\gamma}

\def\whw{\tilde{w}}

\def\cA{{\mathcal A}}
\def\cB{{\mathcal B}}
\def\cC{{\mathcal C}}

\def\cF{{\mathcal F}}

\def\cL{{\mathcal L}}
\def\cM{{\mathcal M}}

\def\cP{{\mathcal P}}

\def\cT{{\mathcal T}}

\def\E{\mathbb{E}}
\def\F{\mathbb{F}}

\def\N{\mathbb{N}}

\def\P{\mathbb{P}}
\def\R{\mathbb{R}}

\def\T{\mathbb{T}}

\def\eps{\epsilon}

\def\d{\mathrm{d}}
\def\bmu{\boldsymbol{\mu}}

\def\xsde{x^*_{\delta,\eps}}
\def\ysde{y^*_{\delta,\eps}}
\def\xsd{x^*_{\delta}}
\def\ysd{y^*_{\delta}}
\def\pde{\Phi_{\delta,\eps}}

\theoremstyle{plain}
\newtheorem{theorem}{Theorem}[section]
\newtheorem{lemma}[theorem]{Lemma}
\newtheorem{corollary}[theorem]{Corollary}
\newtheorem{proposition}[theorem]{Proposition}

\newtheorem{definition}[theorem]{Definition}
\newtheorem{example}[theorem]{Example}
\newtheorem{remark}[theorem]{Remark}

\usepackage[symbol]{footmisc}

\usepackage{fancyhdr}
\date{\vspace{-1em}\normalsize{\today}}

\title{Synchronization in a Kuramoto Mean Field Game}

\author{Rene Carmona\footnote{Department of Operations Research and Financial
Engineering, Princeton University, Princeton, NJ, 08540, USA, email: 
{\tt rcarmona@princeton.edu}. Research of Carmona
 was partially supported by AFOSR FA9550-19-1-0291
 and ARPA-E DE-AR0001289.}
\and Quentin Cormier\footnote{Inria Saclay, France, email:
{\tt quentin.cormier@inria.fr}.}
\and H. Mete Soner\footnote{Department of Operations Research and Financial
Engineering, Princeton University, Princeton, NJ, 08540, USA, email: 
{\tt soner@princeton.edu}. Research of Soner
 was partially supported by the National Science Foundation grant
 DMS 2106462.}}

\date{\today}
\begin{document}
\maketitle
\vspace{5pt}

\abstract{The classical Kuramoto model
is studied in the setting of an infinite
horizon mean field game.  The system is shown to exhibit 
both synchronization and phase transition.
Incoherence below a critical value of the
interaction parameter is demonstrated by
the stability of the
uniform distribution. Above this value, 
the game bifurcates and 
develops self-organizing time homogeneous Nash equilibria. 
As interactions become stronger, 
these stationary solutions become fully synchronized.
Results are proved by an amalgam 
of techniques from nonlinear partial differential
equations, viscosity solutions, stochastic optimal 
control and stochastic  processes.}
\vspace{10pt}
\vspace{10pt}


\section{Introduction}

Originally motivated by systems of chemical  and biological oscillators,
the classical Kuramoto model \cite{Kur} 
has found an amazing range of applications from neuroscience to
Josephson junctions in superconductors, and 
has become a  key mathematical
model to describe self organization in complex systems.
These autonomous oscillators
are coupled through a nonlinear interaction term
which plays a central role in the long time behavior 
of the system.
While the system is unsynchronized
when this term is not sufficiently strong,
fascinatingly they exhibit an abrupt
transition to self organization above a
critical value of the interaction parameter.
Synchronization is an emergent property that occurs in a broad
range of complex systems
such as neural signals, heart beats,
fire-fly lights and circadian rhythms.
Expository papers \cite{ABVRS,S} 
and the references therein provide
an excellent introduction to the model
and its applications.

The analysis of the coupled Kuramoto
oscillators through a mean field game formalism
is  first explored by \cite{YMMS,YMMS2} proving
 bifurcation from incoherence to coordination by a
formal linearization and a  spectral argument.
\cite{CG} further develops this analysis 
in their application  to a jet-lag recovery
model.  We follow these pioneering studies
and analyze the Kuramoto model 
as a discounted infinite horizon stochastic game in the limit 
when the number of oscillators 
goes to infinity. We treat the system of oscillators as 
an infinite particle system, but instead of positing the 
dynamics of the particles, we let the individual particles 
endogenously  determine their behaviors by minimizing a 
cost functional and hopefully, settling in a Nash equilibrium. 
Once the search for equilibrium is recast in this way, 
equilibria are given by solutions of nonlinear systems.
Analytically, they are characterized
by a backward 
dynamic programming equation
coupled to a forward
Fokker-Planck-Kolmogorov equation,
and in  the probabilistic approach, by
forward-backward stochastic differential equations.
Stability analysis of the solutions is delicate because of this 
forward-backward nature of the solution, and to the best of our knowledge, 
it remains a challenging problem.
Except possibly in the finite horizon potential case (cf.~\cite{BC} and the 
references therein) it has not 
been fully addressed in the existing literature on  the subject.
For the stability results of the  Kuramoto model in the classical setting, 
the interested 
reader could consult \cite{HHK,HKR} and the references therein.

With finitely many oscillators, 
we consider the following version of the model 
already introduced in \cite{CG,YMMS2}.
We fix a large integer $N$ and for $i\in\{1,\dots,N\}$ 
let  $\theta^i_t$ be the phase of the $i$-th oscillator
at time $t\ge 0$. 
We assume the phases 
$\theta^i_t$ are controlled Ito diffusion processes
satisfying,
$\d \theta^i_t = \alpha^i_t \d t + \sigma \d B^i_t$,
where $B^i$'s are  
independent Brownian motions,
and the control processes $\alpha^i$
are exerted by the individual 
oscillators so as to 
simultaneously minimize their costs given by
$$ \alpha^i \mapsto J^i(\boldsymbol{\alpha})
:= \E \int_0^\infty e^{-\beta t}\left[ \kappa
\; L(\theta^i_t,\boldsymbol{\theta}_t) + \tfrac{1}{2} (\alpha^i_t)^2 \right] \d t,
$$
where $\boldsymbol{\alpha}=(\alpha^1,\ldots,\alpha^N)$ and 
$\boldsymbol{\theta}_t=(\theta^1_t,\ldots,\theta^N_t)$.
The positive constants $\sigma, \beta$ are 
respectively, the common standard deviations of the 
random shocks affecting the dynamics of the phases, 
and the common discounting factor used to compute the 
present value of the cost.  
The centrally important positive constant $\kappa$ 
models the strength of the interactions between the oscillators.

In line with the classical literature on Kuramoto's synchronization 
theory, we assume that the running
cost function $L$ is given by
\begin{align*}
L(\theta^i,\boldsymbol{\theta}) &
= \frac{1}{N} \sum_{j \neq i}{2 \left(\sin\left((\theta^i - \theta^j)/2 \right)\right)^2 }
= \frac{1}{N} \sum_{j=1}^N{2 \left(\sin\left((\theta^i - \theta^j)/2 \right)\right)^2 }.
\end{align*}
The cost $L$ accounts for the cooperation between the $N$ 
oscillators by incentivizing them to align their frequencies,
while the term $(\alpha^i_t)^2 $ represents a form 
of kinetic energy which is also to be minimized.
It is convenient to 
express the  above cost functional 
by using the empirical distribution measure of the oscillators
as follows,
\begin{equation}
\label{eq:c}
L(\theta^i_t,\boldsymbol{\theta}_t)
= c(\theta^i_t, \bar\mu^N_t),
\qquad
\text{where}
\qquad
c(\theta, \mu) :=
\int { 2\left(\sin\left((\theta - \theta')/2 \right)\right)^2 }\, \mu(\d\theta'),
\end{equation}
and the \emph{empirical measure} $\bar\mu^N_t $ is given by,
$$
\bar\mu^N_t = \bar\mu(\boldsymbol{\theta}_t) 
: =\frac1N\sum_{j=1}^N\delta_{\theta^j_t}.
$$

As the finite particle system is
essentially intractable, especially for large values of $N$,
we follow approach of
\cite{Car,CD,HMC, LL1,LL2,LL} that is now considered standard, and 
approximate the Nash equilibria for the above  
system  of oscillators by letting   
their number $N$ go to infinity.
Then, for a given flow
of probability measures
$\bmu=(\mu_t)_{t \ge 0}$, the stochastic optimal 
control problem for the representative oscillator
is to minimize
\begin{equation}
\label{eq:scp}
\alpha \in \cA \ \mapsto
\E \int_0^\infty e^{-\beta t} 
\left( \ell(t,X_t) + \frac{1}{2} \alpha^2_t \right) \d t, 
\end{equation}
where $\cA$ is the 
set of  all right-continuous and progressively 
measurable processes,
 the \emph{running cost} $\ell(t,x)$
is equal to $\kappa c(x, \mu_t)$ with $c$ as in \eqref{eq:c},
and $X_t$ is the controlled 
phase of the representative oscillator given by
$X_t = X_0 + \int_0^t{  \alpha_u \d u}  + \sigma B_t$,
for a Brownian motion $B_t$.
The Nash equilibrium,
as defined in 
Definition~\ref{def:KMFG}
below, is achieved
when the flow
$\bmu=(\mu_t)_{t \ge0}$
is given by the marginal laws of the optimal process $X^*_t$.  
By direct methods, Lemma~\ref{lem:existence}
proves the existence of such equilibrium flows
starting from any initial
distribution.

It is immediate that the uniform
distribution $U(\d x)=\d x/(2\pi)$
on the torus gives 
a stationary equilibrium flow. Indeed,
$c(x,U)\equiv1$ and therefore, the 
optimal control for the above problem with the 
constant flow $U$ is identically equal to zero. 
As the uniform distribution
has no special structure, 
it represents incoherence among the oscillators,
and when the interaction parameter
is small, we show that all the solutions of the Kuramoto mean field game
converge to this incoherent state.
This global attraction is proved in Lemma~\ref{lem:incoherent}
for $\kappa<\beta \sigma^2/4$.
Theorem~\ref{th:main} 
considers all $\kappa$ 
less than the critical value
\begin{equation}
\label{eq:kc}
\kappa_c := \beta \sigma^2 
+ \sigma^4/2,
\end{equation}
and proves that
there are that start 
``close'' to
the uniform distribution 
converge to it as time tends to infinity.
Thus, Lemma~\ref{lem:incoherent} and
Theorem~\ref{th:main}
reveal that incoherence is the main paradigm 
in the sub-critical regime $\kappa < \kappa_c$.
Theorem~\ref{th:stationary} 
analyzes the case $\kappa>\kappa_c$,
and proves that
there are infinitely many 
self-organizing stationary solutions
for these interaction parameter values. 
In particular, these solutions do not converge to
the incoherent uniform distribution 
and numerically they are stable.
Hence,
$\kappa_c$ is a sharp
threshold for the stability of incoherence,
and there is a phase transition 
from total disorder to self organization
exactly at this critical interaction parameter $\kappa_c$.  Furthermore,
Theorem~\ref{lem:full}  shows convergence 
to full synchronization
as $\kappa$ gets larger. 

The classical Kuramoto model with noise
has been the object of many studies, and
the mean field version is  the following McKean-Vlasov 
stochastic differential equation
$$
\d X_t = - \kappa\ \int_{\T}{ \sin(X_t - y) \cL(X_t)(\d y)}\ \d  t + \sigma \d B_t,
$$
where $\cL(X_t)$ is the law of the random variable $X_t$.
The uniform distribution is shown in \cite{GPP}
to be both locally and globally stable 
when  $\kappa < \sigma^2$. 
The corresponding finite particle system 
is studied in \cite{BGP,Co}.
There, it is proven that the solutions of the finite model
remain close to the solution of 
the above equation for a very long time, on the order 
of $o(\exp(N))$.  Similar results are also proved
for the Kuramoto mean field game with an ergodic cost in 
\cite{YMMS2}, by using bifurcation theory techniques including the 
Lyapunov-Schmidt reduction method
to show the existence of non-uniform stationary solutions
near the critical value 
$\kappa^*_c = \sigma^4/2$.   
Rabinowitz bifurcation theorem
and other global techniques are used in \cite{Ci} 
for similar results. 

The classical Kuramoto model
and its mean field game versions provide a mechanism for the analysis of
self organization. However, they cannot model synchronization
with external drivers, thus requiring
additional terms.  Indeed, the jet-lag recovery model of
\cite{CG}  introduce a cost for misalignment with
the exogenously given sunlight frequency,
providing an incentive to be in synch with
the environment as well.  These studies
are clear evidences of the modeling
potential of the mean field game formalism
in all models when self organization
is the salient feature.
 
The paper is organized as follows. After a short section
on notation, the Kuramoto
mean field game is introduced in Section~\ref{s.KMFG}, and
the main results are stated in Section~\ref{s.main}.
Section~\ref{sec:problems}
briefly summarizes 
all control problems used in the paper. 
Stationary solutions are defined 
and a fixed-point characterization is proved
in Section~\ref{s.char}. The super-critical case is 
studied in Section~\ref{s.sc}
and full synchronization in Section~\ref{s.full}.
Incoherence is demonstrated in Section~\ref{s.incoherent}
by proving the convergence of all solutions
to the uniform distribution when the 
interaction parameter is small,
and local stability of the uniform distribution is 
established in Section \ref{sec:sub}
for all $\kappa <\kappa_c$.
For completeness, solutions
starting from any distribution 
are constructed in the Appendix~\ref{app:exist},
and we provide  
the expected
comparison result for a degenerate
Eikonal equation in the Appendix~\ref{app:compare}.

\section{Notation}
\label{s.notation}
The \emph{state-space} is the one-dimensional
torus $\T := \R / (2 \pi \mathbb{Z})$,
 $\cP(\T)$ is the 
space of all probability measures on $\T$.
For $\nu \in \cP(\T)$, $f \in C(\T)$, we
use the standard notation
$\nu(f):= \int_\T\, f(x)\, \nu(d x).$
We say that a probability measure
$\nu \in \cP(\T)$
is the \emph{law $\cL(X)$ of
$X$}, if
$\E[ f(X)] = \nu(f)$ for every 
$f \in C(\T)$.  We also
use the following space of continuous functions,
$$
\cC:= \{ \ \xi =(\gamma,\eta) : [0,\infty) \mapsto \R^2\ :\ 
\text{continuous and bounded}\ \}.
$$

 We fix a filtered probability space $(\Omega, \F,  \P)$ 
 supporting an $\F$-adapted
 Brownian motion  $(B_t)_{t\ge 0}$. 
 We assume that the filtration $\F=\{\cF_t\}_{t\ge0}$ satisfies the usual 
conditions, i.e. $\cF_0$ is complete and $\cF_t$ is right-continuous.
The initial filtration is non-trivial 
so that for any probability measure $\mu_0 \in \cP(\T)$,
one can construct an $\cF_0$ measurable,
$\T$ valued random variable $X_0$ 
with distribution $\mu_0$. 
For $t\ge 0$, the 
set $\cA_t$ of  all progressively 
measurable processes $\alpha:[t,\infty) \to \R$
is called the
\emph{admissible controls},
and we set $\cA:=\cA_0$.

For $\mu \in \cP(\T), z \in \T$ and a Borel subset $B \subset \T$,
we define the \emph{translation of $\mu$}  by,
\begin{equation}
\label{eq:translate}
\mu(B;x):= \mu(\{ z \in \T\ :\ x+z \in B\ \}).
\end{equation}

Finally, we record several elementary trigonometric identities 
that are used repeatedly.
For $\mu \in \cP(\T)$, 
let $c(x,\mu)$ be as in the Introduction.
As $2\left(\sin\left(x/2 \right)\right)^2 = 1 - \cos(x)$,
\begin{equation}
\label{eq:dcc}
c(x,\mu)= \int_\T 
2\left(\sin\left((x-y)/2 \right)\right)^2\, \mu(\d y) 
=1 - a(\mu) \cos(x) - b(\mu) \sin(x),
\end{equation}
where
$a(\mu):= \mu(\cos)$, and
$b(\mu):=\mu(\sin)$.
In particular, there is $z^* \in \T$ such that
\begin{equation}
\label{eq:abg}
a(\mu(\cdot;z^*))=g(\mu):= \sqrt{(a(\mu))^2+(b(\mu))^2\ },
\qquad
\text{and}
\qquad
b(\mu(\cdot;z^*)) =0.
\end{equation}

\section{Kuramoto mean-field game}
\label{s.KMFG}

Given a flow of probability measures
$\bmu=(\mu_t)_{t \ge 0}$,  set 
$$
\ell_{\bmu}(t,x):= \kappa[c(x,\mu_t) -1]
= - \kappa \mu_t(\cos)\cos(x) - \kappa \mu_t(\sin)\sin(x),
\qquad
x \in \T, t \ge 0.
$$
Consider the optimal control problem
\eqref{eq:scp} with this running cost.
Then, the problem is
\begin{equation}
\label{eq:scp1}
v_{\bmu}:= \inf_{\alpha \in \cA}  J_{{\bmu},\kappa}(\alpha)
:= \inf_{\alpha \in \cA} \E \int_0^\infty e^{-\beta t} 
\left( \ell_{\bmu}(t,X^\alpha_t) + \frac{1}{2} \alpha^2_t \right) \d t, 
\end{equation}
where as in the Introduction,
$X^\alpha_t := X_0 + \int_0^t{  \alpha_u \d u}  + \sigma B_t$, 
with a Brownian motion $(B_t)_{t \geq 0}$
and an initial condition
$X_0$ satisfying $\cL(X_0)=\mu_0$.  
\vspace{5pt}

\begin{definition}
\label{def:KMFG}{\rm{
We say that $\bmu= (\mu_t)_{t \geq 0}$ 
is a solution to the}} Kuramoto mean-field game
with interaction parameter $\kappa$
starting from initial distribution $\mu_0$,
\rm{if there exists $\alpha^* \in \cA$ 
such that} 
$J_{{\bmu},\kappa}(\alpha^*) = \inf_{\alpha \in \cA} 
J_{{\bmu},\kappa}(\alpha)$ and
$\mu_t = \cL(X^{\alpha^*}_t)$ for all $t\ge 0$.
\end{definition} 

\begin{example}
\label{ex:example}{\rm{
Consider an initial condition 
$X_0$ 
satisfying
$\E \cos(X_0) = \E \sin(X_0) = 0$,
and the flow of probability measures 
$\bmu=(\mu_t)_{t \geq 0}$ with
$
\mu_t := \cL(X_0 + \sigma B_t).
$
Then, for every $t\ge0$,
\begin{align*}
\ell_{\bmu}(t, x)
&=-\kappa(\mu_t(\cos)\cos(x) +\mu_t(\sin)\sin(x))\\ &=
-\kappa(\ \cos(x) \E [\cos(X_0 + \sigma B_t)] + 
\sin(x) \E [\sin(X_0 + \sigma B_t) ])\\
&= -\kappa(\cos(x)  e^{-\frac{\sigma^2}{2}t } \E [\cos(X_0)]
- \sin(x) e^{-\frac{\sigma^2}{2}t }  \E [\sin(X_0)])= 0.
\end{align*}
Therefore, for any $\alpha \in \cA$, 
$J_{{\bmu},\kappa}(\alpha)  =  \E \int_0^\infty e^{-\beta t}\ 
\frac{\alpha^2_t }{2} \ \d t \ge 0 = J_{{\bmu},\kappa}(0)$,
implying that $\alpha^* \equiv 0$ is the 
minimizer of $J_{{\bmu},\kappa}(\alpha)$, and $\bmu$ 
is the law of the dynamics controlled by $\alpha^*$. Hence,
$\bmu$ is a solution of the 
Kuramoto mean-field game for every $\kappa$.

\vskip 2pt
Now suppose that $\cL(X_0)$ is 
the uniform probability measure on the torus $U(\d x)=\d x/(2\pi)$.
As any translation of $U$ is equal to itself,
$\mu_t=\cL(X_0+\sigma B_t) = U$ for all $t\ge0$.  Thus, $U$
is a stationary 
solution.}}  \qed
\end{example}

The uniform distribution
represents complete incoherence, and
we refer to it as the \emph{incoherent} (or \emph{uniform}) solution.
We next introduce the stationary solutions of the 
Kuramoto mean-field games.
\begin{definition}{\rm{
We call a probability measure $\mu \in \cP(\T)$}} 
a stationary solution {\rm{if the constant flow 
$\boldsymbol{\mu} = (\mu_t)_{t \geq 0}$  with $\mu_t = \mu$
for all $t \geq 0$ 
is a solution of the Kuramoto mean-field game.  
We say that $\mu$ is}} self-organizing or non-uniform
{\rm{if it is not equal to the uniform measure $U$.}}
\end{definition}
\vspace{5pt}

We record the following simple result for future reference.

\begin{lemma}
\label{lem:uniform0}
The uniform probability measure $U$ on the torus
is the incoherent stationary solution of the Kuramoto mean-field game.
Moreover, a stationary solution $\mu$ is 
the uniform probability measure 
if and only if
$\mu(\cos)=\mu(\sin)=0$.
\end{lemma}
\begin{proof}
In Example~\ref{ex:example},
we have shown that $U$ is a stationary solution
 and that $c(\cdot,U)\equiv1$.
Now suppose that $\mu$ is a stationary solution with
$\mu(\cos)=\mu(\sin)=0$.  Then, as in Example~\ref{ex:example},
we conclude that the optimal solution of the control problem
\eqref{eq:scp} is $\alpha^*\equiv 0$, and 
 the optimal state process
satisfies $\d X^*_t= \sigma \d B_t$. As by stationarity
$\cL(X^*_t)= \mu$ for every $t \ge 0$, 
the density $f$ of $\mu$ solves the 
Fokker-Plank equation $f_{xx}(x)=0$ on the torus.
Hence, $f$ is equal to a constant, and $\mu =U$.
\end{proof}

\begin{remark}[Invariance by translation]
\label{rk:inv}
{\rm{
Assume that $\mu$ is a  stationary solution. 
The symmetry of the problem implies that
the translated measure
$\mu(\cdot;z) $ is also a stationary solution for every $z$.
 }}
\qed
\end{remark}

\section{Main Results}
\label{s.main}
In this section, we state all the
main results of the paper.
Recall the critical interaction parameter $\kappa_c$ of \eqref{eq:kc}.
In Section \ref{s.sc}, we study the super-critical case  $\kappa>\kappa_c$,
and prove the following result.

\begin{theorem}[Super-critical interaction: synchronization]
\label{th:stationary}
For  all  interaction parameters $\kappa > \kappa_c $,
there are  non-uniform  stationary solutions of the 
Kuramoto mean field game.
\end{theorem}

Suppose $\mu$ is one of the non-uniform stationary solutions  given by the above result.
Then, any translation  $\mu(\cdot; z)$ is
also a stationary solution. 
We conjecture that up to these translations,
there exists a unique non-uniform stationary
solution of the Kuramoto mean-field
game for every interaction parameter
$\kappa > \kappa_c$, see Remark~\ref{rem:conjectures} below.

We interpret these non-uniform stationary solutions as partially organized states 
of the Kuramoto mean-field
game, and conclude that for interaction parameters $\kappa$ larger than the critical value $\kappa_c$,
there is self organization.  
As $\kappa$ gets larger
the stationary measure become
more localized and Theorem \ref{lem:full},
proved in Section \ref{s.full} below,
shows convergence to the fully
synchronized regime corresponding to
stationary Dirac measures.

\begin{theorem}[Strong interaction: full synchronization]
\label{lem:full}
Let $\mu_n$ be a 
sequence of non-uniform stationary solutions
of  the Kuramoto mean-field game
with interaction parameters $\kappa_n$
 tending to infinity.  Then, 
there exists a sequence $z_n\in \T$ such that 
the translated stationary solutions 
$\mu_n(\cdot \ ; z_n)$ converge in law
to the Dirac measure $\delta_{\{0\}}$.
\end{theorem}
\noindent

\vskip 2pt
We have already argued in Example~\ref{ex:example}
that the uniform measure is always a 
stationary solution for all interaction parameters.
In Section \ref{s.incoherent} below,
we consider small interaction parameters 
and prove that all solutions converge to
this incoherent state.

\begin{lemma}[Weak interaction: incoherence]
\label{lem:incoherent}
If  $\kappa < \beta \sigma^2/4$,
then any solution $\bmu=(\mu_t)_{t \ge 0}$
of the Kuramoto mean field game with interaction
parameter $\kappa$ converges to
the incoherent state, i.e., 
as $t$ tends to infinity,
$\mu_t$ converges in law to $U$.
\end{lemma}

The next result, proved in Section~\ref{sec:sub}, addresses the 
local stability for all $\kappa<\kappa_c$,
showing that a phase 
transition occurs \emph{exactly} at $\kappa_c$.
This result require the initial distribution $\mu_0$
to be sufficiently close 
to the uniform distribution. To quantify
the distance of any measure $\mu_0 \in \cP(\T)$
to the uniform measure, we set
\begin{equation}
\label{eq:d}
d(\mu_0) :=
\max \left\{ |\mu_0(\cos)|\, ,\, 
|\mu_0(\sin)|\, ,\, 
|\mu_0(\sin\cos)|\, ,\, 
|\mu_0(\cos^2)-\tfrac12|\, 
\right\}.
\end{equation}
\begin{theorem}[Sub-critical  interaction: desynchronization]
\label{th:main}

For $\kappa < \kappa_c$, there is
$c_\kappa$  
such that for every $\mu_0 \in \cP(\T)$
satisfying 
$d(\mu_0) \le c_\kappa$,
there exists a solution $\bmu^*=(\mu^*_t)_{t \ge 0}$
of the Kuramoto mean field game with interaction parameter
$\kappa$ and initial distribution $\mu_0$,
such that $\mu^*_t$ converges in law to the 
uniform distribution
as $t$ tends to infinity. Moreover, this convergence is exponential 
in the sense that for some $\lambda^*_\kappa>0$,
\begin{equation}
\label{eq:exponential rate convergence th}
\sup_{t \geq 0 }\, e^{\lambda^*_\kappa t}\, d(\mu^*_t)  < \infty.
\end{equation} 
\end{theorem}

The existence of 
solutions to mean field games is well
known for problems with ergodic cost \cite{LL1,LL2,LL}.
However, for discounted infinite horizon 
problems it follows directly from
our general approach.  Thus, 
we provide this proof  for completeness in the Appendix~\ref{app:exist}.
\begin{lemma}[Existence of solutions]
\label{lem:existence}
For any probability measure $\mu_0 \in \cP(\T)$
and $\kappa \ge 0$,
there exists a solution $\bmu=(\mu_t)_{t \ge 0}$ of
the Kuramoto mean field game
with interaction parameter $\kappa$
starting from initial distribution $\mu_0$.
\end{lemma}

\subsection{Illustration of the results}
\label{ss.illustrate}

To illustrate our main results
numerically, we consider the problem with 
parameters $\beta = 1/2$, $\sigma = 1$
with critical value $\kappa_c = 1$.
We numerically compute 
the solutions of the Kuramoto mean field
game with two interaction parameters.

The first case $\kappa=0.8$ is
below the threshold, and  we are in the regime
considered in 
Theorem~\ref{th:main}.
We compute a solution with initial condition 
$ \nu(\d x) = C \exp\left(- \sin(x)\right)\d x$. 
Left panel in Figure~\ref{fig:simudensity} illustrates the convergence
of the solution  to the uniform
distribution. 

The case $\kappa=2$ is above the 
critical value and Theorem \ref{th:stationary}
implies that there are non-uniform
stationary solutions.  Indeed,
we compute a solution of the Kuramoto
mean field game
with initial  distribution 
that have two clusters around $\pi/2$ and $3 \pi /2$,
$$
\nu(\d x) = C\, \chi_{[\pi/4, \pi/4 + \pi/10] \cup [\pi, \pi + \pi/10]}(x)\, \d x.
$$
As seen in the right panel of Figure~\ref{fig:simudensity},
the two clusters quickly merge and the solution converges 
towards a non-uniform invariant probability measure, 
whose shape is reported Figure~\ref{fig:fig1}.
\vspace{5pt}

\begin{figure}[ht]
\centering
\includegraphics[width=0.4\linewidth]{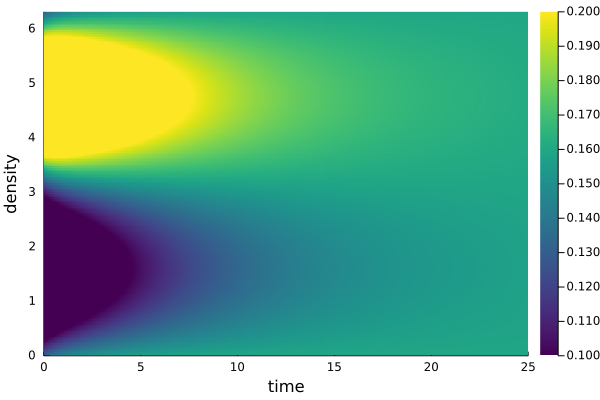} 
\includegraphics[width=0.4\linewidth]{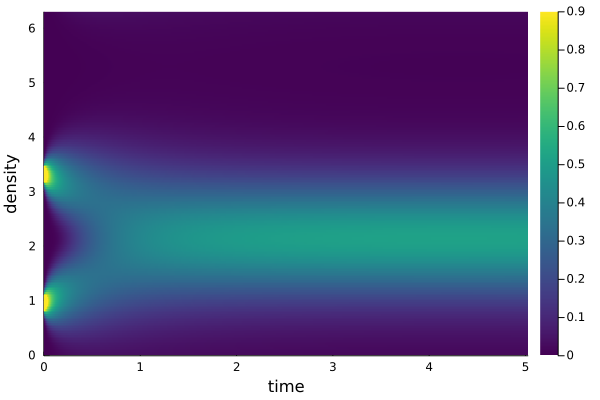} 
\caption{Left panel sub critical interaction, 
right panel super critical.}
\label{fig:simudensity}
\end{figure}

In all our numerical experiments
with $\kappa >\kappa_c$,
the solutions converge to shifts
of the solutions constructed
in the proof of Theorem~\ref{th:stationary}.
The exact translation is determined by the initial distribution.  
We do not provide a study
of this interesting phenomenon. 

\section{Control problems}
\label{sec:problems}

The original and the central 
stochastic optimal control problem 
is defined in \eqref{eq:scp1}.  However, 
in the sequel, we use several other closely related
problems in our analysis.  
So to highlight the 
subtle differences among them and to
provide a general overview of the notation, we define
all of them in this section.  It is also clear that adding
a constant to the running cost of any control problem does not alter
the minimizing control. As we are only 
interested in the optimal behavior, we use
this flexibility and appropriately
modify the problem whenever it is 
convenient. 

\subsection{Inhomogeneous problems}
\label{ss:inhomogenous}

For $\xi= (\gamma,\eta) \in \cC$, we consider 
the stochastic control
\begin{equation}
\label{eq:scp2}
v_{\xi}(\mu_0):= \inf_{\alpha \in \cA}  J_{\xi}(\mu_0,\alpha)
:= \inf_{\alpha \in \cA} \E \int_0^\infty e^{-\beta t} 
\left( \ell_{\xi}(t,X^\alpha_t) + \frac{1}{2} \alpha^2_t \right) \d t, 
\end{equation}
where $X^\alpha_t := X_0 + \int_0^t{  \alpha_u \d u}  + \sigma B_t$ is as before 
with initial data satisfying $\cL(X_0)=\mu_0$,
and the running
cost is given by,
\begin{equation}
\label{eq:ell}
\ell_\xi(t,x)= -\gamma(t) \cos(x) -\eta(t) \sin(x),
\qquad
x \in \T, t \ge 0.
\end{equation}  
We let $X^\xi$ be the optimal state process.
The dependence on $\mu_0$ through
the condition $\cL(X^\xi_0)=\mu_0$ 
is omitted in the notation for simplicity.
To characterize the dynamics of $X^\xi$,
we also need to introduce
a family of control
problems starting from any pair $(t,x) \in [0,\infty) \times \T$.
Recall that $\cA_t$ is the set
of all adapted control process $\alpha:[t,\infty) \mapsto \R$.
We set
\begin{equation}
\label{eq:vx}
v^\xi(t,x):= \inf_{\alpha \in \cA_t} J_\xi(t,x,\alpha)
:= \inf_{\alpha \in \cA_t}\E \int_t^\infty e^{-\beta(u-t)} 
[\ell_\xi(u,X^{\alpha,(t,x)}_u) +\tfrac12 \alpha_u^2] \ \d u,
\end{equation}
where
\begin{equation}
\label{eq:xalpha}
X^{\alpha,(t,x)}_u = x + \int_t^u{\alpha_s \d s} + \sigma [B_u-B_t],
\qquad u \ge t. 
\end{equation}
We use the notation $X^{\alpha,x}=X^{\alpha,(0,x)}$.
For a given process $\xi$, let $\alpha^*$ be the optimal control
with initial data $(t,x)$ and let
$X^{(t,x),\xi}= X^{\alpha^*,(t,x)}$ be the optimal state
process making the dependence on the process $\xi$
explicit. 

\subsection{Stationary problem}
\label{ss:stationary}
When the flow $\bmu$ is given by
one probability measure $\mu \in \cP(\T)$,
we obtain a stationary problem. 
The corresponding value
function is given by,
\begin{equation}
\label{eq:vkappa}
v_{\mu,\kappa}(x):= \inf_{\alpha \in \cA} J_{\mu,\kappa}(x,\alpha)
:= \inf_{\alpha \in \cA} 
\E \int_0^\infty e^{-\beta t} [\ell_{\mu,\kappa}(X^{\alpha, x}_t) +\tfrac12 \alpha_t^2] \ \d t,
\end{equation}
where as before $\ell_{\mu,\kappa}(x):= \kappa[c(x,\mu)-1] = - \kappa[
\mu(\cos)\cos(x) + \mu(\sin)\sin(x)]$.

\subsection{Parametrized problems}
\label{ss:parametrize}
 
Similarly, we may consider
functions $\xi \in \cC$ that are  time-homogeneous.
Additionally, in this case we 
can translate the corresponding measure appropriately
so that the second component is zero.
So we only use the first component  $\gamma \in \R$
and  let $\ell_\gamma(x):= -\gamma\cos(x)$. We then
set
\begin{equation}
\label{eq:value}
v^\gamma(x):= \inf_{\alpha \in \cA} J_\gamma(x,\alpha)
:= \inf_{\alpha \in \cA} 
\E \int_0^\infty e^{-\beta t} [\ell_\gamma(X^{\alpha,x}_t) +\tfrac12 \alpha_t^2] \ \d t.
\end{equation}
We further
elaborate on this problem in Section~\ref{s.char} below.

\section{Stationary solutions}
\label{s.char}

In this section,  
we establish a 
one-to-one correspondence between the 
stationary solutions and fixed
points of a scalar function of one-variable
that we construct.

\subsection{System of partial differential equations}
\label{ss.fbs}

It is well-known that the solutions
of mean-field games can be obtained
by solving a system of coupled 
partial differential equations, \eqref{eq:hjb} and \eqref{eq:fk} in the present situation.
Indeed, the dynamic programming 
(Hamilton-Jacobi-Bellman)
equation related to the stochastic 
optimal control problem \eqref{eq:scp}
with  any time-homogeneous running cost $\ell$
 is given by,
\begin{equation}
\label{eq:hjb}
\beta v(x) -\frac{\sigma^2}{2} v_{xx}(x) +\frac12\ (v_x(x))^2
= \ell(x).
\end{equation}
For smooth $\ell$,
the above equation has classical solutions
(cf.~Lemma~\ref{lem:smooth} below) and
the solution is the value function given by
\eqref{eq:value} with running cost $\ell$.
Moreover, the optimal feedback control is 
 $\alpha^*(x)= - v_x(x)$, and  
 the optimal state process
solves 
$\d X^*_t = - v_x(X^*_t) \d t + \sigma \d B_t$.
The stationary law of
 $X^*_t$ has a density $f$ that solves
the following stationary Fokker-Plank equation,
\begin{equation}
\label{eq:fk}
\partial_x \left( v_x(x) f(x)
+ \frac{\sigma^2}{2} f_{x}(x) \right)= 0.
\end{equation}
The unique solution $f_v$  of the above 
equation is explicitly available, 
cf.~\eqref{eq:mugamma}.

\begin{remark}
\label{rem:initial}
{\rm{We emphasize that the initial
condition $X^*_0$ of the optimal process 
is random and its density 
is given by $f_v$.  In particular,
the density of $X^*$
is also a part of the solution.
This is in contrast with the 
time-varying problems
\eqref{eq:scp1} and \eqref{eq:scp2},
in which the initial distribution $\mu_0$
is given and the solutions depend
on $\mu_0$.

\qed}}
\end{remark}

Recall  the value function $v_{\mu,\kappa}$ 
of \eqref{eq:vkappa}, and 
let $f_{v_{\kappa,\mu}}$ be the solution
of \eqref{eq:fk} with this value function.
The following characterization
follows directly from these definitions.

\begin{lemma}
\label{lem:kuramoto}
A probability measure $\mu$ is a stationary 
solution of the Kuramoto mean field game with
interaction parameter $\kappa$, if and only if
its density  is equal to $f_{v_{\kappa,\mu}}$.
\end{lemma}

We close this subsection 
with another simple result  reported for completeness.

\begin{lemma}
\label{lem:smooth}
For $\beta > 0$ and $\ell \in \mathcal{C}^1(\T)$,
there exists a unique solution $v \in \mathcal{C}^2(\T)$ 
of \eqref{eq:hjb}.  Moreover, when $\ell$ is even so is $v$.
\end{lemma}
\begin{proof}
Existence of a unique 
smooth solution of \eqref{eq:hjb}
 is classical in the elliptic regularity theory \cite{GT}.
Suppose that $\ell$ is even and let $v$ be the unique solution.
Set $\widehat{v}(x):=v(-x)$.  It is clear that $\widehat{v}$
also solves \eqref{eq:hjb}.  Thus, by uniqueness $v=\widehat{v}$.
\end{proof}

\subsection{Characterization}
\label{ss.char}

Using the system of differential
equations \eqref{eq:hjb} and \eqref{eq:fk},
we establish a 
one-to-one correspondence  between the 
stationary solutions and fixed
points of a scalar-valued function of one-variable.
For $\gamma \ge 0$, let $v^\gamma$
be as in \eqref{eq:value} and set
$\mu^\gamma(\d x) = f_{v^\gamma}(x)\d x$. Then,
the solution is explicitly given by,
\begin{equation}
\label{eq:mugamma}
\mu^\gamma(\d x) = \frac{1}{Z^\gamma}\,  
\exp \left( - \frac{2}{\sigma^2} v^\gamma(x) \right)\, \d x,
\qquad
Z^\gamma\, = \int_\T
\exp \left( - \frac{2}{\sigma^2} v^\gamma(y) \right)\, \d y.
\end{equation}
For $\kappa \ge0$, set
\begin{equation}
\label{eq:F}
F_\kappa(\gamma) := \kappa\ \mu^\gamma(\cos),
\qquad \gamma \ge 0,
\end{equation}
\vspace{5pt}

\noindent
Note that for $\gamma=0$, the measure $\mu^0$ is the uniform measure as $v^\gamma(x)\equiv 0$, and therefore, 
$\gamma=0$ is a fixed point of the function $F_\kappa$ for every $\kappa$. The case $\gamma>0$ is treated next.
For the following discussion, 
recall $a(\mu), b(\mu)$ of \eqref{eq:dcc},
$g(\mu)$ of \eqref{eq:abg}, and
$\mu(\cdot;z)$  of \eqref{eq:translate}.  

\begin{proposition}
\label{p.char}
A probability measure  $\mu \in \cP(\T)$ is a 
non-uniform stationary solution of the
Kuramoto mean-field game with an interaction
parameter $\kappa$, if and only if 
$\kappa g(\mu)$ is a strictly positive
fixed point of $F_\kappa$
and $\mu^{\kappa g(\mu)}=\mu(\cdot;z)$ for some 
$z \in \T$.  Moreover, if $\gamma \in (0,\kappa]$ is a
fixed point of $F_\kappa$, then $\mu^\gamma$
is a non-uniform stationary solution.
\end{proposition}

The above result also implies that the existence of
non-uniform stationary solutions is equivalent to the existence of 
positive fixed points of $F_\kappa$.  
\begin{proof}
Suppose that $\mu$ is a stationary solution.
By Remark \ref{rk:inv}, any
translation $\mu(\cdot; z)$ is again a stationary
solution.
Choose $z \in \T$ as in \eqref{eq:abg}
so that $b(\mu(\cdot;z))=0$, and
$a(\mu(\cdot;z))=g(\mu) \ge 0$.
Set, $\gamma:= \kappa\, a(\mu(\cdot;z))
= \kappa\, g(\mu)$.
By  \eqref{eq:dcc}, 
$$
\ell_{\mu(\cdot;z),\kappa}(x):= \kappa[c(x,\mu(\cdot;z))-1] =
 - \gamma \cos(x).
 $$
Then, the value function $v_{\kappa,\mu(\cdot;z)}$ 
of \eqref{eq:vkappa}, and $v^\gamma$ of \eqref{eq:value}
are equal.
Moreover, as $\mu(\cdot;z)$ is a stationary
solution,
it is equal to the law of the 
optimal state process of this problem.
Therefore, its density 
is equal to
the solution $f_{v^\gamma}$
of the Fokker-Plank equation \eqref{eq:fk}.
Hence, $\mu(\cdot;z)=\mu^\gamma$.
By its definition
$F_\kappa(\gamma)= \kappa\, a(\mu^\gamma)$, 
and by our choice $\gamma:= \kappa\, a(\mu(\cdot;z))$.
Hence, $\gamma$ is a fixed point of $F_\kappa$.

To prove the opposite implication, assume that
$\gamma:=\kappa g(\mu)$ is a fixed point of $F_\kappa$.
By Lemma~\ref{lem:smooth}, $v^\gamma$ and therefore
the density of $\mu^\gamma$ are even.  This implies that
$b(\mu^\gamma)=0$.  Also
$\gamma = F_\kappa(\gamma)= \kappa\, a(\mu^\gamma)$.
Hence by \eqref{eq:dcc}, 
$\ell_{\mu^\gamma,\kappa}= - \gamma \cos(x)$
and  $v_{\kappa,\mu^\gamma}$ of \eqref{eq:vkappa}
is equal to $v^\gamma$ of \eqref{eq:value}.
Hence, by Lemma~\ref{lem:kuramoto}, $\mu^\gamma$
is a stationary solution.

Moreover, choose $z \in \T$ as in \eqref{eq:abg}
so that $\kappa\, c(x,\mu(\cdot;z)) =   \kappa - \kappa g(\mu) \cos(x)$.
Since by definition $\gamma = \kappa g(\mu)$,
the control problems \eqref{eq:scp}
for the stationary flows $\mu(\cdot;z)$ and $\mu^\gamma$
are the same.  Hence, $\mu(\cdot;z)$ is equal  $\mu^\gamma$
with $\gamma = \kappa g(\mu)$.

Finally suppose that $\gamma\in (0,\kappa]$ be a fixed point of
$F_\kappa$.  In the above we have shown that $\mu^\gamma$
is a stationary solution.  Moreover, 
$\kappa \mu^\gamma(\cos)=F_\kappa(\gamma)=\gamma \neq 0$.
Hence,
by Lemma~\ref{lem:uniform0}, $\mu^\gamma$ is non-uniform.

\end{proof}

\section{Partial self-organization}
\label{s.sc}

In this section,  we use
the characterization
obtained in the previous section
to prove the existence of non-uniform (or self-organizing)
stationary 
solutions for super-critical parameters,
proving Theorem~\ref{th:stationary}.
Towards this goal, we analyze the function
$F_\kappa$ defined by \eqref{eq:F} near the origin and at infinity.
A numerical example of $F_\kappa$  with $\kappa >\kappa_c$
is given in Figure~\ref{fig:fig1} below. 

\captionsetup[subfigure]{labelformat=empty}
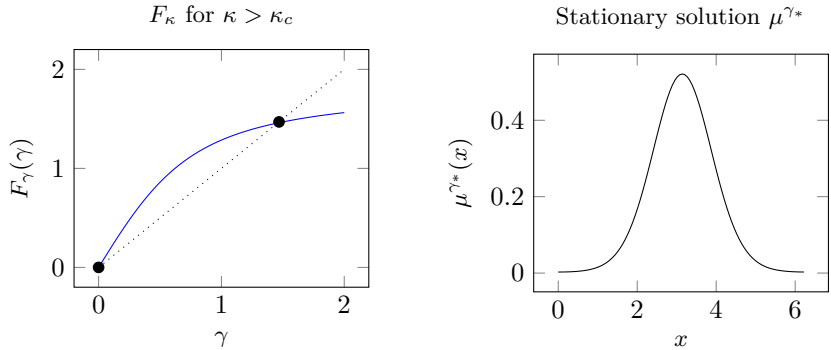
\begin{figure}[ht]
	\centering
	\subfloat[]{
		\begin{tikzpicture}[scale=1.0]
			\begin{axis}[xlabel={\small{$\gamma$}}, ylabel={\small{$F_\gamma(\gamma)$}}, 
			title={\small{$F_\kappa$ for $\kappa > \kappa_c$}}, 
				 width=5.5cm]
				\addplot[blue]
				coordinates {
					(0.0,1.3877787807814457e-17)
					(0.04081632653061224,0.08155104457896493)
					(0.08163265306122448,0.1625356524241261)
					(0.12244897959183673,0.24240601497370523)
					(0.16326530612244897,0.3206502334327408)
					(0.20408163265306123,0.39680707436204143)
					(0.24489795918367346,0.47047732281518334)
					(0.2857142857142857,0.5413312409282955)
					(0.32653061224489793,0.6091120414995093)
					(0.3673469387755102,0.6736356458051322)
					(0.40816326530612246,0.734787269656267)
					(0.4489795918367347,0.7925155501871314)
					(0.4897959183673469,0.8468249879956627)
					(0.5306122448979592,0.8977674511860316)
					(0.5714285714285714,0.9454333949279998)
					(0.6122448979591837,0.9899433199782748)
					(0.6530612244897959,1.031439850917669)
					(0.6938775510204082,1.0700806784832193)
					(0.7346938775510204,1.1060324923881717)
					(0.7755102040816326,1.1394659373147664)
					(0.8163265306122449,1.1705515564286795)
					(0.8571428571428571,1.1994566417050558)
					(0.8979591836734694,1.2263428847507538)
					(0.9387755102040817,1.251364711229239)
					(0.9795918367346939,1.2746681821651697)
					(1.0204081632653061,1.2963903526174931)
					(1.0612244897959184,1.3166589894999106)
					(1.1020408163265305,1.335592563487705)
					(1.1428571428571428,1.3533004434441898)
					(1.183673469387755,1.3698832346422662)
					(1.2244897959183674,1.3854332136790903)
					(1.2653061224489797,1.4000348231189992)
					(1.3061224489795917,1.4137651974885679)
					(1.346938775510204,1.426694699352092)
					(1.3877551020408163,1.4388874499526865)
					(1.4285714285714286,1.450401843482132)
					(1.469387755102041,1.4612910376200274)
					(1.510204081632653,1.4716034157289533)
					(1.5510204081632653,1.4813830181623557)
					(1.5918367346938775,1.4906699416702234)
					(1.6326530612244898,1.4995007069888617)
					(1.6734693877551021,1.507908595470543)
					(1.7142857142857142,1.515923956123636)
					(1.7551020408163265,1.5235744847567627)
					(1.7959183673469388,1.530885477100613)
					(1.836734693877551,1.5378800578569722)
					(1.8775510204081634,1.5445793876263303)
					(1.9183673469387754,1.5510028496156774)
					(1.9591836734693877,1.5571682179444493)
					(2.0,1.5630918092614259)
				}
				;
				\addplot[dotted]
				coordinates {
					(0.0,0.0)
					(0.020202020202020204,0.020202020202020204)
					(0.04040404040404041,0.04040404040404041)
					(0.06060606060606061,0.06060606060606061)
					(0.08080808080808081,0.08080808080808081)
					(0.10101010101010101,0.10101010101010101)
					(0.12121212121212122,0.12121212121212122)
					(0.1414141414141414,0.1414141414141414)
					(0.16161616161616163,0.16161616161616163)
					(0.18181818181818182,0.18181818181818182)
					(0.20202020202020202,0.20202020202020202)
					(0.2222222222222222,0.2222222222222222)
					(0.24242424242424243,0.24242424242424243)
					(0.26262626262626265,0.26262626262626265)
					(0.2828282828282828,0.2828282828282828)
					(0.30303030303030304,0.30303030303030304)
					(0.32323232323232326,0.32323232323232326)
					(0.3434343434343434,0.3434343434343434)
					(0.36363636363636365,0.36363636363636365)
					(0.3838383838383838,0.3838383838383838)
					(0.40404040404040403,0.40404040404040403)
					(0.42424242424242425,0.42424242424242425)
					(0.4444444444444444,0.4444444444444444)
					(0.46464646464646464,0.46464646464646464)
					(0.48484848484848486,0.48484848484848486)
					(0.5050505050505051,0.5050505050505051)
					(0.5252525252525253,0.5252525252525253)
					(0.5454545454545454,0.5454545454545454)
					(0.5656565656565656,0.5656565656565656)
					(0.5858585858585859,0.5858585858585859)
					(0.6060606060606061,0.6060606060606061)
					(0.6262626262626263,0.6262626262626263)
					(0.6464646464646465,0.6464646464646465)
					(0.6666666666666666,0.6666666666666666)
					(0.6868686868686869,0.6868686868686869)
					(0.7070707070707071,0.7070707070707071)
					(0.7272727272727273,0.7272727272727273)
					(0.7474747474747475,0.7474747474747475)
					(0.7676767676767676,0.7676767676767676)
					(0.7878787878787878,0.7878787878787878)
					(0.8080808080808081,0.8080808080808081)
					(0.8282828282828283,0.8282828282828283)
					(0.8484848484848485,0.8484848484848485)
					(0.8686868686868687,0.8686868686868687)
					(0.8888888888888888,0.8888888888888888)
					(0.9090909090909091,0.9090909090909091)
					(0.9292929292929293,0.9292929292929293)
					(0.9494949494949495,0.9494949494949495)
					(0.9696969696969697,0.9696969696969697)
					(0.98989898989899,0.98989898989899)
					(1.0101010101010102,1.0101010101010102)
					(1.0303030303030303,1.0303030303030303)
					(1.0505050505050506,1.0505050505050506)
					(1.0707070707070707,1.0707070707070707)
					(1.0909090909090908,1.0909090909090908)
					(1.1111111111111112,1.1111111111111112)
					(1.1313131313131313,1.1313131313131313)
					(1.1515151515151516,1.1515151515151516)
					(1.1717171717171717,1.1717171717171717)
					(1.1919191919191918,1.1919191919191918)
					(1.2121212121212122,1.2121212121212122)
					(1.2323232323232323,1.2323232323232323)
					(1.2525252525252526,1.2525252525252526)
					(1.2727272727272727,1.2727272727272727)
					(1.292929292929293,1.292929292929293)
					(1.3131313131313131,1.3131313131313131)
					(1.3333333333333333,1.3333333333333333)
					(1.3535353535353536,1.3535353535353536)
					(1.3737373737373737,1.3737373737373737)
					(1.393939393939394,1.393939393939394)
					(1.4141414141414141,1.4141414141414141)
					(1.4343434343434343,1.4343434343434343)
					(1.4545454545454546,1.4545454545454546)
					(1.4747474747474747,1.4747474747474747)
					(1.494949494949495,1.494949494949495)
					(1.5151515151515151,1.5151515151515151)
					(1.5353535353535352,1.5353535353535352)
					(1.5555555555555556,1.5555555555555556)
					(1.5757575757575757,1.5757575757575757)
					(1.595959595959596,1.595959595959596)
					(1.6161616161616161,1.6161616161616161)
					(1.6363636363636365,1.6363636363636365)
					(1.6565656565656566,1.6565656565656566)
					(1.6767676767676767,1.6767676767676767)
					(1.696969696969697,1.696969696969697)
					(1.7171717171717171,1.7171717171717171)
					(1.7373737373737375,1.7373737373737375)
					(1.7575757575757576,1.7575757575757576)
					(1.7777777777777777,1.7777777777777777)
					(1.797979797979798,1.797979797979798)
					(1.8181818181818181,1.8181818181818181)
					(1.8383838383838385,1.8383838383838385)
					(1.8585858585858586,1.8585858585858586)
					(1.878787878787879,1.878787878787879)
					(1.898989898989899,1.898989898989899)
					(1.9191919191919191,1.9191919191919191)
					(1.9393939393939394,1.9393939393939394)
					(1.9595959595959596,1.9595959595959596)
					(1.97979797979798,1.97979797979798)
					(2.0,2.0)
				}
				;
				\addplot[only marks]
				coordinates {
					(0, 0)
					(1.469,1.469)
				};
			\end{axis}
		\end{tikzpicture}
	}
	\qquad
	\subfloat[]{
		\begin{tikzpicture}[scale=1]
			\begin{axis}[xlabel={\small{$x$}}, ylabel={\small{$\mu^{\gamma_*}(x)$}}, 
					title style={align=center},
				title={\small{Stationary solution $\mu^{\gamma_*}$}},
		 width=5.5cm]
				\addplot[]
				coordinates {
					(0.0,0.002734012533656596)
					(0.06283185307179587,0.002759374956466707)
					(0.12566370614359174,0.002836377097826817)
					(0.18849555921538758,0.002967791551592939)
					(0.25132741228718347,0.003158332749616995)
					(0.3141592653589793,0.0034147977427779566)
					(0.37699111843077515,0.003746264335786855)
					(0.43982297150257105,0.004164347158939643)
					(0.5026548245743669,0.004683511761133989)
					(0.5654866776461628,0.005321445793569692)
					(0.6283185307179586,0.006099484694022787)
					(0.6911503837897545,0.007043086848198664)
					(0.7539822368615503,0.008182349875459587)
					(0.8168140899333463,0.009552555358836045)
					(0.8796459430051421,0.011194723948257913)
					(0.9424777960769379,0.013156156305198613)
					(1.0053096491487339,0.015490927906115263)
					(1.0681415022205296,0.01826029747637885)
					(1.1309733552923256,0.021532980126634762)
					(1.1938052083641213,0.025385227623268457)
					(1.2566370614359172,0.02990065034951559)
					(1.3194689145077132,0.035169709307261814)
					(1.382300767579509,0.041288803058849737)
					(1.4451326206513049,0.048358875043172295)
					(1.5079644737231006,0.05648347252080759)
					(1.5707963267948966,0.06576620077030375)
					(1.6336281798666925,0.07630753615134761)
					(1.6964600329384882,0.08820098999450085)
					(1.7592918860102842,0.10152865214850622)
					(1.82212373908208,0.11635618785320113)
					(1.8849555921538759,0.13272741294150403)
					(1.9477874452256718,0.15065862772209573)
					(2.0106192982974678,0.1701329457244731)
					(2.0734511513692633,0.19109490531214793)
					(2.1362830044410592,0.21344569476786426)
					(2.199114857512855,0.23703934924421383)
					(2.261946710584651,0.26168028551352024)
					(2.324778563656447,0.28712252305151087)
					(2.3876104167282426,0.3130708943649692)
					(2.4504422698000385,0.3391844723747193)
					(2.5132741228718345,0.3650823393948062)
					(2.5761059759436304,0.39035169497589367)
					(2.6389378290154264,0.4145581556742779)
					(2.701769682087222,0.43725794833418563)
					(2.764601535159018,0.45801155134484156)
					(2.827433388230814,0.4763982081508101)
					(2.8902652413026098,0.49203063642837774)
					(2.9530970943744057,0.5045691956314438)
					(3.015928947446201,0.5137347630930678)
					(3.078760800517997,0.5193196087166968)
					(3.141592653589793,0.5211956501179461)
					(3.204424506661589,0.5193196087166968)
					(3.267256359733385,0.5137347630930673)
					(3.3300882128051805,0.5045691956314433)
					(3.3929200658769765,0.4920306364283771)
					(3.4557519189487724,0.4763982081508094)
					(3.5185837720205684,0.45801155134484095)
					(3.5814156250923643,0.43725794833418524)
					(3.64424747816416,0.41455815567427756)
					(3.707079331235956,0.3903516949758933)
					(3.7699111843077517,0.3650823393948059)
					(3.8327430373795477,0.3391844723747191)
					(3.8955748904513436,0.31307089436496893)
					(3.958406743523139,0.28712252305151054)
					(4.0212385965949355,0.26168028551352)
					(4.084070449666731,0.23703934924421363)
					(4.146902302738527,0.21344569476786396)
					(4.209734155810323,0.19109490531214768)
					(4.2725660088821185,0.1701329457244728)
					(4.335397861953915,0.15065862772209562)
					(4.39822971502571,0.13272741294150375)
					(4.461061568097506,0.11635618785320093)
					(4.523893421169302,0.101528652148506)
					(4.586725274241098,0.08820098999450066)
					(4.649557127312894,0.0763075361513474)
					(4.71238898038469,0.06576620077030358)
					(4.775220833456485,0.05648347252080746)
					(4.838052686528282,0.04835887504317219)
					(4.900884539600077,0.041288803058849646)
					(4.9637163926718735,0.03516970930726174)
					(5.026548245743669,0.02990065034951551)
					(5.0893800988154645,0.025385227623268402)
					(5.152211951887261,0.021532980126634717)
					(5.215043804959056,0.018260297476378817)
					(5.277875658030853,0.015490927906115231)
					(5.340707511102648,0.013156156305198587)
					(5.403539364174444,0.011194723948257894)
					(5.46637121724624,0.00955255535883603)
					(5.529203070318036,0.008182349875459573)
					(5.592034923389832,0.007043086848198653)
					(5.654866776461628,0.006099484694022778)
					(5.717698629533423,0.005321445793569683)
					(5.7805304826052195,0.004683511761133983)
					(5.843362335677015,0.004164347158939639)
					(5.906194188748811,0.003746264335786851)
					(5.969026041820607,0.003414797742777953)
					(6.031857894892402,0.003158332749616993)
					(6.094689747964199,0.002967791551592937)
					(6.157521601035994,0.002836377097826816)
					(6.220353454107791,0.002759374956466706)
				}
				;
			\end{axis}
		\end{tikzpicture}
	}
	\caption{
		For parameters $\beta = 1/2$, $\sigma = 1$, $\kappa = 2 \kappa_c = 2$,  the function
		$F_\kappa$ has two fixed points 
		at $\gamma = 0$ and $\gamma_* \approx 1.47$.}
	\label{fig:fig1}
\end{figure}

Set $A(\gamma):= a(\mu^\gamma)$ so that  by \eqref{eq:F},
$F_\kappa(\gamma)=\kappa A(\gamma)$. 

\begin{lemma}
\label{lem:derivative}
The function $A$ defined above
is differentiable at the origin and $A'(0) =1/\kappa_c$.
In particular, $F_\kappa'(0)>1$ for all $\kappa > \kappa_c$
and  there is $\gamma_0>0$ such that
\begin{equation}
\label{eq:big}
F_\kappa(\gamma) > \gamma, 
\qquad \forall\, \kappa \ge 2\kappa_c,\ \gamma \in (0,\gamma_0).
\end{equation} 
\end{lemma}
\begin{proof}
As $v^\gamma$ solves \eqref{eq:hjb}
with $\ell(x)=-\gamma \cos(x)$,
$k(x):=(v^\gamma)_x(x)$
solves
$$
\beta k(x)
- \frac{\sigma^2}{2} k_{xx} (x)
+\, (v^\gamma)_x(x)\, k_x (x)
=\gamma \sin(x).  
$$
Since $|\sin(x)|\le1$, by maximum principle,
we conclude that $|k(x)|  \le (\gamma/\beta)$.
Next consider
$h(x) := v^\gamma(x) - u(x)$ 
where $u(x) = - (2 \gamma \cos(x))/(2 \beta + \sigma^2)$.
Since $u$ 
solves the linear equation
$\beta u(x)-\frac{\sigma^2}{2} u_{xx}(x) =-\gamma\, \cos(x)$, 
$h$ 
satisfies
$$
\beta h(x)
- \frac{\sigma^2}{2} h_{xx}(x) = 
-\frac{1}{2} k^2(x). 
$$
By Feynman–Kac,
$$
|h(x)| =  \frac12
|\int_0^\infty{ e^{-\beta t} \E \left[  k^2(x + \sigma B_t)  \right] dt }|
\le \frac{\gamma^2}{2 \beta^3}.  
$$
Summarizing, we have shown that
$$
v^\gamma(x)= -\frac{2 \gamma \cos(x)}{2\beta + \sigma^2}
+ h(x),
\quad
\text{and}
\quad
|h(x)| \le  
\frac{\gamma^2}{2 \beta^3} .
$$
As a result, $\lim_{\gamma\downarrow 0}v^\gamma(x)=0$ and by dominated convergence, 
$\lim_{\gamma\downarrow 0}Z^\gamma=2\pi$.
Therefore, for any $y \in \T$,
\begin{align*}
\lim_{\gamma \to 0}\frac{1}{\gamma}\ 
[\exp(- \frac{2}{\sigma^2}\, v^\gamma(y)) -1]
&= \lim_{\gamma \to 0}\frac{1}{\gamma}\ 
\left[\exp ( - \frac{2}{\sigma^2}\, [-\frac{2\gamma \cos(y)}{2\beta + \sigma^2}
+h(y)] ) -1\right]\\
&= \lim_{\gamma \to 0}\frac{1}{\gamma}\ 
[\exp ( - \frac{2}{\sigma^2}\, [-\frac{2\gamma \cos(y)}{2\beta + \sigma^2}] ) -1]
= \frac{4}{\sigma^2(2\beta + \sigma^2)}\,  \cos(y).
\end{align*}
Moreover, there exists a constant $c$ such that
$| \exp(- \frac{2}{\sigma^2}\, v^\gamma(y)) -1| \le c \gamma$ for every
$y \in \T$. 
We also directly calculate that 
$$
Z_0 = 2\pi,
\qquad
\lim_{\gamma \to 0} F_\kappa(0)=0.
$$
Hence, by the dominated convergence theorem
and the above calculations,
\begin{align*}
A'(0)&= \lim_{\gamma \to 0}
\frac{A(\gamma)}{\gamma}
=\lim_{\gamma \to 0}\ \frac{1}{Z^\gamma}\,
\int_{-\pi}^\pi \cos(y)\, \frac{1}{\gamma}\, \exp(- \frac{2}{\sigma^2}
\, v^\gamma(y))\d y\\
&=\frac{1}{2\pi} \, \lim_{\gamma \to 0}\,
\int_{-\pi}^\pi \cos(y)\, 
\frac{1}{\gamma}\, [\exp(- \frac{2}{\sigma^2}\, v^\gamma(x)) -1]\d y\\
&=\frac{1}{2\pi} \, \frac{4}{\sigma^2(2\beta + \sigma^2)}\,
\int_{-\pi}^\pi \cos^2(y)\, \d y
=\frac{2}{\sigma^2(2\beta + \sigma^2)}.
\end{align*}

To prove the final statement,
we choose $\gamma_0>0$ so that
$A(\gamma) \ge 2\gamma /(3\kappa_c)$
for all $\gamma \in [0,\gamma_0]$.
Then, for $\kappa\ge 2 \kappa_c$,
$F_\kappa(\gamma)=\kappa A(\gamma)
\ge 4\gamma/3 >\gamma$ for all
$\gamma \in [0,\gamma_0]$.

\end{proof}

We continue with an easy upper bound.
\begin{lemma}[Upper bound]
\label{lem:upperbound}
$$
v^\gamma(x) \leq -\frac{\gamma}{\beta} + \sqrt{\gamma\ }\, [
\frac{x^2}{2} + \frac{\sigma^2}{2 \beta}],
\qquad \forall\ x \in [-\pi, \pi].
$$
\end{lemma}
\begin{proof}
Since $ - \cos(x)  \le-1 + x^2/2$,
$v^\gamma \leq  \bar{v}^\gamma$, where
$$
\bar{v}^\gamma(x) 
:=  \inf_{\alpha \in \cA} \E \int_0^\infty{ e^{-\beta t} 
(-\gamma + \frac12 \left[\alpha^2_t +  \gamma (X^{x,\alpha}_t)^2\right])\,  \d t}. 
$$
This linear quadratic stochastic optimization problem has an explicit solution given by
$$
\bar{v}^\gamma(x) =-\frac{\gamma}{\beta}+\sqrt{\gamma\ }[a x^2+ b],
$$
where $a = -\frac{\beta}{4 \sqrt{\gamma}} + \frac{1}{4} \sqrt{\frac{\beta^2}{\gamma} + 4} \le \frac12$,
and $ b=a \sigma^2/\beta \le \sigma^2/(2\beta)$.
\end{proof}

\begin{lemma}
\label{lem:continuous}
$F_\kappa$ is continuous on $\mathbb{R}_+$.
\end{lemma}
\begin{proof}
Fix $\gamma, \delta \in \mathbb{R}$ and define
$u(x) := v^{\gamma+\delta}(x) - v^\gamma(x)$.  
In view of \eqref{eq:hjb},
$$
\beta u(x) -  \frac{\sigma^2}{2} u_{xx}(x) 
+ \frac{1}{2} u_x(x) (v^{\gamma+\delta}_x(x) + v^{\gamma}_x(x) ) 
= - \delta \cos(x).
$$
By maximum principle, we conclude that $\|u\|_\infty \le |\delta|/\beta$.
Additionally, the above upper bound 
implies imply that $\gamma \mapsto \| v^\gamma \|_\infty$ 
is bounded on bounded sets.
Hence, by an application of  the dominated convergence theorem, 
we conclude that $F_\kappa$ is continuous.

\end{proof}
\vspace{5pt}

\noindent
{\emph{Proof of Theorem \ref{th:stationary}}.
By its definition $|F_\kappa(\gamma)| \le \kappa$.
Moreover, $F_\kappa$ is continuous on $\mathbb{R}_+$, 
and is differentiable at $\gamma = 0$ 
with $F_\kappa'(0) = \kappa/\kappa_c$. 
Therefore, for $\kappa > \kappa_c$, 
$F'_\kappa(0)>1$ and consequently, 
$F_\kappa$ has a second fixed point $\gamma_* > 0$.
Since $F_\kappa$ is bounded by $\kappa$, 
$\gamma_* \le \kappa$.
By Proposition \ref{p.char},
$\mu_{\gamma^*}$ is a non-uniform
stationary solution of the 
Kuramoto mean-field game
with interaction parameter $\kappa$. 

\qed

\begin{remark} 
\label{rem:conjectures}
{\rm{
In our numerical experiments,
we always obtained a 
\emph{concave} function $F_\kappa$ 
as depicted in the Figure~\ref{fig:fig1},
and observed that  in the super-critical case,
self-organizing stationary
solutions are unique up to translations.
Moreover, 
time inhomogeneous solutions
converge to a  stationary
solution. 
We thus conjecture that
the function $F_\kappa$
is concave for every interaction parameter
and that non-uniform solutions
are unique. Concavity
would also imply that this unique
stationary measure converges to 
the uniform measure as $\kappa\downarrow \kappa_c$.
A complete analysis
of these observations and the conjecture would be 
highly interesting.}}
\end{remark}

\section{Full synchronization: \texorpdfstring{$\kappa \to \infty$}{K -> Infty}}
\label{s.full}

In this section, we 
prove Theorem~\ref{lem:full}.
An important step is the following
result.
\vspace{5pt}

\begin{proposition}
\label{p:limit}
As $\gamma $ tend to infinity,
$\mu^\gamma$ converges in law to the Dirac measure $\delta_{\{0\}}$. 
\end{proposition}
\vspace{5pt}

The above result follows from several lemmas 
proved in the next subsection~\ref{ss.infinity}.  
\vspace{5pt}

\noindent
\emph{Proof of Theorem~\ref{lem:full}.}.  Let $\mu_n$
and $\kappa_n$ be as in the statement of the theorem.
Choose $z_n$ as in \eqref{eq:abg}. Then 
by Proposition~\ref{p.char},
$\gamma_n = \kappa_n  g(\mu_n)= \kappa_n a(\mu_n(\cdot;z_n))$
is a fixed point of $F_{\kappa_n}$ and 
$\mu_n(\cdot;z_n)=\mu^{\gamma_n}$.
We claim that $\lim_{n \to \infty} \gamma_n = \infty$.  If this claim holds,
then by Proposition~\ref{p:limit}, we conclude that
$\mu_n(\cdot;z_n)$ converges in law to $\delta_{\{0\}}$,
completing the proof of Theorem~\ref{lem:full}.

We continue by proving our claim that 
$\lim_{n \to \infty} \gamma_n = \infty$,
by a counter argument.  So we assume that
on a subsequence $\gamma_n$ remains bounded.
Without loss of generality, we take the 
subsequence to be the whole sequence.  Then,
 on a subsequence, denoted by $n$ again,
$\gamma_n$ converges to $\gamma^*$.
It is clear that $v^{\gamma_n}$ converges to
$v^{\gamma^*}$ and consequently,
$\mu^{\gamma_n} $
converges in law to $\mu^*:=\mu^{\gamma^*}$.
Also, for sufficiently
large $n$,  $\kappa_n \ge 2\kappa_c$ and in view 
of \eqref{eq:big}, the fixed point $\gamma_n$
of $F_{\kappa_n}$ is larger than $\gamma_0$.
So we conclude that the limit point $\gamma^* \ge \gamma_0>0$.

Summarizing  $v^*:= v^{\gamma^*}$
is the value function of \eqref{eq:scp} with running cost
$\ell(x)=-\gamma^* \cos(x)$.  The 
stationary law of the optimal state process is $\mu^*=\mu^{\gamma^*}$.
Furthermore,
$$
a(\mu^*)= \lim_{n \to \infty}\, a(\mu^{\gamma_n}) =
\lim_{n \to \infty}\,   \frac{F_{\kappa_n}(\gamma_n)}{\kappa_n}=
\lim_{n \to \infty}\,  \frac{\gamma_n}{\kappa_n} =0.
$$
Also by \eqref{eq:mugamma}, 
\begin{align*}
a(\mu^*) &= \int \cos(x) \mu^*(\d x) =
\frac{1}{Z_{\gamma^*}}\,
\int_{-\frac{\pi}{2}}^{\frac{3\pi}{2}} \cos(x) \exp{(-\frac{2}{\sigma} v^*(x))}\ \d x\\
&= \frac{1}{Z_{\gamma^*}}\, [
\int_{-\frac{\pi}{2}}^{\frac{\pi}{2}} \cos(x) \exp{(-\frac{2}{\sigma} v^*(x))}\ \d x
+ \int_{\frac{\pi}{2}}^{\frac{3\pi}{2}} \cos(x) \exp{(-\frac{2}{\sigma} v^*(x))}\ \d x\, ]\\
&= \frac{1}{Z_{\gamma^*}}\, [
\int_{-\frac{\pi}{2}}^{\frac{\pi}{2}} \cos(x) \exp{(-\frac{2}{\sigma} v^*(x))}\ \d x
+ \int_{-\frac{\pi}{2}}^{\frac{\pi}{2}} \cos(x+\pi) \exp{(-\frac{2}{\sigma} v^*(x+\pi))}\ \d x\, ]\\
&= \frac{1}{Z_{\gamma^*}}\,
\int_{-\frac{\pi}{2}}^{\frac{\pi}{2}} 
 \cos(x)  \exp{(-\frac{2}{\sigma} v^*(x))}
 \left[1- \exp{(\frac{2}{\sigma} [v^*(x)
- v^*(x+\pi)])}\right]\ \d x.
\end{align*}
By Lemma~\ref{lem:value} below, 
we conclude that 
the above integral is strictly
positive, which is in contradiction
with the fact that $a(\mu^*)=0$.

\qed

The following lemma is used in the above proof.  Set
$$
\widehat{w}(x):= v^\gamma(x) -v^\gamma(x+\pi),
\qquad x \in [-\frac{\pi}{2},\frac{\pi}{2}].
$$

\begin{lemma}
\label{lem:value}
For $x,y, \in \R$, $\gamma>0$, if $\cos(x) \ge \cos(y)$,
then $v^\gamma(x) \le v^\gamma(y)$.  In particular,
$\widehat{w}(x) \le 0$ on $[-\frac{\pi}{2},\frac{\pi}{2}]$,
and it is not identically equal to $0$.
\end{lemma}

\begin{proof} Let $v^\gamma$ be as in \eqref{eq:value}. 
For any
stopping time $\tau$,
dynamic programming implies that
$$
v^\gamma(x) = \inf_{\alpha \in \cA} \, J_\gamma(x,\tau,\alpha),
$$
where
$$
J_\gamma(x,\tau,\alpha) := 
\E\int_0^\tau e^{-\beta t}\, 
[ - \gamma \cos(X^{\alpha,x}_t)+ \frac12 \alpha_t^2]\, \d t
+ e^{-\beta \tau}\, v^\gamma(X^{\alpha,x}_\tau),
$$
as in Section~\ref{sec:problems}, 
$X^{\alpha,x}_t=x +\int_0^t \alpha_u \d u + \sigma B_t$.
In particular,
\begin{equation}
\label{eq:13}
v^\gamma(x) -v^\gamma(y)
\le \sup_{\alpha\in \cA} [ J_\gamma(x,\tau,\alpha) 
-J_\gamma(y,\tau,\alpha) ].
\end{equation}
First suppose that $\cos(x) = \cos(y)$.  Then, either $x=y$ 
or $x=-y$. As $v^\gamma$ is even, in either case
$v^\gamma(x) = v^\gamma(y)$.

We now fix $x,y$ such that $\cos(x) > \cos(y)$
and consider the stopping time
$$
\tau = \inf\{ t >0\ :\
\cos(X^{\alpha,x}_t)=\cos(X^{\alpha,y}_t)\ \}.
$$
Then, $\cos(X^{\alpha,x}_\tau)=\cos(X^{\alpha,y}_\tau)$,
and consequently, 
$v^\gamma(X^{\alpha,x}_\tau)= v^\gamma(X^{\alpha,y}_\tau)$.
Since for every $ t \in [0,\tau)$, 
$\cos(X^{\alpha,x}_t)>\cos(X^{\alpha,y}_t)$
we have
$$
J_\gamma(x,\tau,\alpha)  -J_\gamma(y,\tau,\alpha) \le 0.
$$
Hence, by \eqref{eq:13} 
we conclude that $v^\gamma(x) \le v^\gamma(y)$.

Since for $x \in [-\frac{\pi}{2},\frac{\pi}{2}]$, 
$\cos(x) \ge \cos(x+\pi)$, this implies that
$\widehat{w}(x) \le 0$.  Suppose that 
$v^\gamma(0)=v^\gamma(\pi)$.
As $\cos(0) \ge \cos(x) \ge \cos(\pi)$ for
every $x \in [0,\pi]$, this implies that $v^\gamma \equiv v^\gamma(0)$.
However, $v^\gamma$ is a classical solution of \eqref{eq:hjb}
with $\ell(x)=- \gamma \cos(x)$ and a constant function
is not a solution of this equation.  So we
conclude that $\widehat{w}(0)<0$.

\end{proof}

\subsection{Proof of Proposition~\ref{p:limit}}
\label{ss.infinity}

The central analytical object in this proof is the following scaled function
\begin{equation}
\label{eq:wgamma1}
w^{\gamma}(x) := \frac{1}{\sqrt{\gamma\ }}\,  [v^{\gamma}(x) + \frac{\gamma}{\beta}],
\end{equation}
which solves the  equation
\begin{equation}
\label{eq:wgamma}
\beta_\gamma w^\gamma(x)
- \frac{\sigma_\gamma^2}{2} w_{xx}^\gamma(x)
+\frac{1}{2} (w_x^{\gamma}(x))^2  
= 1 - \cos(x)
= 2 (\sin(x/2))^2,
\end{equation}
where $\sigma_\gamma := \sigma \gamma^{-1/4}$,
and  $\beta_\gamma :=  \beta \gamma^{-1/2}$.
Then, $w^\gamma$ has the following
stochastic optimal control representation,
$$
w^\gamma(x) = \inf_{\alpha \in \mathcal{A}_0} 
L_\gamma(x,\alpha) := \E 
\int_0^\infty{ 
e^{-\beta_\gamma t} \left[\frac{\alpha^2_t}{2} + 1 
- \cos(X^{\alpha,x}_t) \right] \d t},
$$
with $X^{\alpha,x}_t := x + \int_0^t{\alpha_u \d u} + \sigma_\gamma B_t$. 
Moreover, by Lemma~\ref{lem:upperbound},
$$
|w^\gamma(x) | \le \frac{x^2}{2} + \frac{\sigma^2}{2 \beta}
\le \frac{\pi^2}{2} + \frac{\sigma^2}{2 \beta},
\qquad \forall\, x \in [-\pi,\pi].
$$

We start with a uniform Lipschitz estimate.

\begin{lemma}
\label{lem:lip}
For all $\gamma > 0$,
$$
|w^\gamma_x(x)| \leq  \frac{1}{2}[3 + 
2 \pi +  \pi^2+ \frac{\sigma^2}{\beta} ],\qquad
\forall \ x \in [0,2\pi].
$$
\end{lemma}
\begin{proof}
Let $J_\gamma(z,\alpha)$ be as above.
Fix  $x, x'$, $\eps>0$ 
and  choose an $\eps$-optimal control $\ae \in \cA$ 
satisfying
$\jg(x,\ae) \le \wg(x) +\eps$.
Define a new control $\ap$ by,
$$
\ap_t := 
\begin{cases} \ae_t + (x-x'), &\quad \text{ if } t \leq 1, \\
\ae_t, &\quad \text{ if } t > 1.	
\end{cases}
$$
It is clear that $\ap \in \cA$
and also we have the following,
$$
X^{ \ap,x'}_t - X^{\ae,x}_t = 
\begin{cases}
(x'-x)(1-t), &\quad \text{ if } t \leq 1, \\
0,&\quad \text{ if } t > 1.	
\end{cases}
$$
This implies that
$$
L_\gamma(x', \ap) - L_\gamma(x, \ae) = 
\E \int_0^1{ e^{-\beta_\gamma}(A_t + B_t) \d t}, 
$$
where $A_t := \cos(X^{\ae,x}_t ) - \cos(X^{\ap,x' }_t)$ and 
$B_t := \frac{1}{2} \left[ \ae_t + (x-x') \right]^2 -  \frac{1}{2} (\ae_t)^2$.
Therefore, $|A_t| \leq |x-x'|$, and
$$
|B_t| \le \ae_t |x-x'|  + \frac{(x-x')^2}{2} 
 \leq  \frac{1}{2}[  |x-x'|\,(\ae_t)^2  +  |x-x'| +(x-x')^2 ].
$$
These imply that
\begin{align*}
w^\gamma(x') - w^\gamma(x)-\eps &\leq L_\gamma(x', \ap) - L_\gamma(x, \ae)\\
&\le \E \int_0^1{ e^{-\beta_\gamma}(A_t + B_t) \d t}\\
&\le  \frac{1}{2}[ 3|x-x'| + |x-x'|^2] +  |x-x'| \, 
\E \int_0^1{ e^{-\beta_\gamma}\frac12  (\ae_t)^2 \d t}\\
&\le  \frac{1}{2}[3 + |x-x'| + \pi^2+ |x-x'| \,\jg(x,\ae) \\
&\le  \frac{1}{2}[3 + |x-x'| + \pi^2+ |x-x'| \,[w^\gamma(x) +\eps] \\
&\le  \frac{1}{2}[3 + |x-x'| + \pi^2+ \frac{2 \sigma^2}{\beta} +\eps ]\, |x-x'|.
\end{align*}
 As the argument is symmetric
in $x,x'$ and $\eps$ is arbitrary,  the proof of the lemma is complete.

\end{proof}

Above estimates imply that 
$(w^\gamma )_{\gamma > 0}$ is equicontinuous and uniformly 
bounded.
Hence, by Arzel\`a–Ascoli,  it converges 
uniformly on subsequences.
We continue by identifying the 
limit of $w^\gamma(\cdot) - w^\gamma(0)$ 
which is sufficient for our purposes.
We achieve this 
by using standard tools from
the theory of viscosity solution \cite{CL,CEL,FS}.
\begin{proposition}
\label{prop:wgamma}
As $\gamma$ tends to infinity,
$w^{\gamma}(\cdot) - w^{\gamma}(0)$ converges uniformly
to $w$ given by,
\begin{equation}
\label{eq:explicit}
w(x) := 4\, (1-| \cos(x/2)|).
\end{equation}
In particular,
the function $w$ is the unique viscosity solution 
of the Eikonal equation
\begin{equation}
	\label{eq:eikonaleq} 
	\frac{1}{2} (w_x)^2 = 2(\sin(x/2))^2, \quad  x \in (0,2\pi),
	\text{ with } \quad w(0) = w( 2\pi) = 0.
\end{equation}
\end{proposition}
\begin{proof}
Observe that $w^\gamma$ is a classical and hence, a viscosity solution
of \eqref{eq:wgamma}. As $\beta_\gamma, \sigma_\gamma$ 
converge to zero, the equation \eqref{eq:wgamma}
formally converges  to the
Eikonal equation \eqref{eq:eikonaleq}.
Then, by the classical stability results for viscosity solutions
(cf.~Theorem 1.4 in \cite{CEL} or 
Lemma II.6.2 in \cite{FS}),
imply that any uniform limit $w$
of  $w^\gamma(\cdot) - w^\gamma(0)$
 is a viscosity solution of \eqref{eq:eikonaleq}.
 We also directly verify that $w$
 defined above is a viscosity solution of \eqref{eq:eikonaleq}.
 By the standard 
 comparison result
 for this equation  (proved in 
 Lemma~\ref{lem:compare} for completeness), 
 we conclude that,
 any uniform limit of $w^\gamma(\cdot) - w^\gamma(0)$
 is equal to $w$.

\end{proof}

\noindent
\emph{Proof of Proposition \ref{p:limit}.}
Set
$\whw_\gamma := \tfrac{2}{\sigma^2} (w^\gamma(x) - w^\gamma(0))$,
so that by \eqref{eq:wgamma1},
$$
v^\gamma(x) = \frac{\sqrt{\gamma\ }  \sigma^2}{2} \whw_\gamma 
- \frac{\gamma}{\beta} + \sqrt{\gamma\ }\ w^\gamma(0).
$$
The definition of $\mu^\gamma$ implies that
$$ 
\mu^\gamma(\d x) = 
\frac{ \exp\left(-\sqrt{\gamma\ }\  \whw_\gamma \right) \d x}{  
	\int_{-\pi}^{\pi} 
	\exp\left(-\sqrt{\gamma\ }\  \tilde{w}_\gamma(y) \right) \d y } 
	=: \frac{1 }{ \tilde{Z}_\gamma }   
	\exp\left(-\sqrt{\gamma\ }\  \whw_\gamma \right) \d x. 
$$
By  Proposition \ref{prop:wgamma},
$\whw_\gamma$ converges to
$\whw(x)  :=  \frac{8}{\sigma^2}(1-|\cos(x/2)|)$.
Hence,
$\whw_\gamma(x) = \tilde{w}(x) + \epsilon_\gamma(x)$,
for some function $\epsilon_\gamma$ converging uniformly to zero 
as $\gamma$ tends to infinity.  
It is clear that on $x \in [-\pi,\pi]$,
$\whw$ is strictly 
convex and has a unique global minimum 
$\whw(0)=0$. Therefore, there exists a 
constant $c_1$
such that
$$
\int_{-\pi}^{\pi}\ \chi_{ \{  \whw(y) \leq \frac{1}{\sqrt{\gamma}} \}}\  \d y \geq 
c_1 \gamma^{-1/4}. 
$$
Fix $\epsilon > 0$. There is $\gamma_\eps$ such 
that for all $\gamma \ge \gamma_\eps$,
we have $\|\epsilon_\gamma\|_\infty \leq \epsilon$.
Therefore, for $\gamma \ge \gamma_\eps$,
\begin{align*}
\tilde{Z}_\gamma &= 
\int_{-\pi}^{\pi} \exp\left(-\sqrt{\gamma\ } [  \tilde{w}(y) + \epsilon_\gamma(y) ] \right)\ \d y \\
& \geq \exp(- \epsilon \sqrt{\gamma\ }) 
\int_{ \{    \tilde{w} \leq \frac{1}{\sqrt{\gamma}} \} } 
\exp(- \sqrt{\gamma} \tilde{w}(y))\ \d y \\
& \geq  \exp(-  [\epsilon \sqrt{\gamma\ } +1] )\ c_1 \gamma^{-1/4}
=: (c_0\, e^{\eps \sqrt{\gamma\ }} \gamma^{1/4})^{-1}.
\end{align*}
As $\|\epsilon_\gamma\|_\infty \leq \epsilon$ 
for all $\gamma \ge \gamma_\eps$, for these values of $\gamma$ the 
following estimate holds
\begin{align*}
\int_{-\pi}^{\pi}\ \chi_{ \{ \whw \geq 3\epsilon\}} \mu^\gamma(\d x) 
&= \frac{1}{\tilde{Z}_\gamma}  
\int_{-\pi}^{\pi}\ \chi_{ \{ \whw \geq 3\epsilon\}} 
e^{-\sqrt{\gamma\ } \whw_\gamma }\ \d x \\
& \leq c_0 \, e^{\eps \sqrt{\gamma\ }}\, \gamma^{1/4} \
e^{ \sqrt{\gamma\ } \|\epsilon_\gamma\|_\infty  }
\int_{-\pi}^{\pi}\ \chi_{\{\whw \geq 3\epsilon\}} 
e^{-\sqrt{\gamma\ } \whw(x) } \ \d x\\
&\leq 2\pi\ c_0\,   \gamma^{1/4}  e^{-\epsilon  \sqrt{\gamma\ }  }.
\end{align*}
Since the above quantity converges to zero as $\gamma$
tends to infinity, we conclude that any
limit point of 
$\mu^\gamma$ does not have any mass in the
set $\{\whw \geq 3\epsilon\}$ for every $\eps$.   
This set shrinks
to the singleton $\{0\}$ as $\eps$ tends to zero.
Hence, $\mu^\gamma$
converges in  law to $\delta_{\{0\}}$.

\qed

\section{Weak interaction and incoherence}
\label{s.incoherent}

In this section we consider small $\kappa$ values,
and  prove the convergence of
all solutions to the Kuramoto mean field game
to the uniform solution as time gets larger.
In the next section, we consider
all $\kappa<\kappa_c$, and 
prove the existence of convergent
solutions provided that initial 
distribution is sufficiently close
to the uniform distribution.

\subsection{Setting}
\label{ss.spaces}

For a continuous function $\xi=(\gamma,\eta) \in \cC$
and a probability measure $\mu_0 \in \cP(\T)$,
recall the value function $v^\xi(t,x)$ of \eqref{eq:vkappa},
running cost $\ell_\xi$ of \eqref{eq:ell},
the state processes 
$X^{\alpha, (t,x)}$ of \eqref{eq:xalpha},
and the optimal state process
$X^{\xi,(t,x)}$  of \eqref{eq:value},
$X^\xi$ of the problem \eqref{eq:scp2}
with initial distribution $\cL(X^\xi_0)=\mu_0$
(the dependence on $\mu_0$ is omitted in the notation
for simplicity).

It is well-known \cite{FS}  that the value function
$v^\xi(t,x)$ of \eqref{eq:vx} is a classical solution of the 
time inhomogeneous dynamic programming
equation,
\begin{equation}
\label{eq:dpe}
\beta v^\xi(t,x) 
-v^\xi_t(t,x) 
- \frac{\sigma^2}{2}\ v^\xi_{xx}(t,x)
+\frac12 (v^\xi_x(t,x))^2 
=\ell_\xi(t,x),
\quad 
t>0, x \in \T.
\end{equation}
Then, the
optimal state process $X^\xi$
is given by
\begin{equation}
\label{eq:optsde}
\d X^\xi_t = -\, v^\xi_x(u,X^\xi_u)\ \d u + \sigma \d B_t,
\end{equation}
with initial data $\cL(X_0)=\mu_0$.

We now define a map
\begin{equation} 
\label{eq:cT}
\cT (\cdot;\mu_0) : \xi \in \cC \ \mapsto \ 
\cT(\xi;\mu_0):=(\E[\cos(X^\xi_t)], \E[\sin(X^\xi_t)])_{t \ge 0}.
\end{equation}
For a given probability flow $\bmu=(\mu_t)_{t \ge 0}$,
define
\begin{equation}
\label{eq:xi}
\xi(\bmu)=(\mu_t(\cos),\mu_t(\sin))_{t \ge 0}.
\end{equation}
The following is an immediate consequence 
of the definitions.
\begin{lemma}
\label{lem:fixed}
A  probability flow $\bmu=(\mu_t)_{t \ge 0}$
is a solution of the Kuramoto mean field game 
if and only if $\xi(\bmu)$ defined
in \eqref{eq:xi}  is a fixed point of 
$\kappa \cT(\cdot;\mu_0)$. Moreover, if 
$\xi$ is a fixed point of $\kappa \cT(\cdot;\mu_0)$, 
then the probability flow
$(\cL(X^\xi_t))_{t \ge 0}$ is a solution of
the Kuramoto mean field game
starting from the distribution $\mu_0$.
\end{lemma}

\subsection{Estimates}
\label{ss.estimates}

For any function $k:[0,\infty)\times \T \mapsto \R^d$
and $t \ge 0$,
we set
$$
\|k\|_{t,\infty}:= \sup_{u \ge t}\ \|k(u,\cdot ) \|_\infty.
$$

\begin{lemma}
\label{lem:vx}
For any $\xi \in \cC$, $\kappa >0$ and $t\ge 0$,
$\| v^\xi_x \|_{t,\infty} \le \tfrac{1}{\beta}\, \|\xi\|_{t,\infty}$.
\end{lemma}

\begin{proof}
For any $x,y \in \T$ and $t \ge 0$,
\begin{align*}
v^\xi(t,x) - v^\xi(t,y) 
&\le \sup_{\alpha \in \cA_t} [J_{\xi}(t,x,\alpha) -J_\xi(t,y,\alpha)]\\
&\le \sup_{\alpha \in \cA_t} 
\E \int_t^\infty e^{-\beta(u-t)} 
[\ell_\xi(u,X^{\alpha,(t,x)}_u) -\ell_\xi(u,X^{\alpha, (t,y)}_u)] \ \d u,  
\end{align*}
where $\ell_\xi$ is as in \eqref{eq:ell}.  Then,
$$
\left| \ell_\xi(u,X^{\alpha,(t,x)}_u) -\ell_\xi(u,X^{\alpha,(t,y)}_u)
\right| \le \|\xi\|_{t,\infty}\, |x-y|, \qquad \forall\, u \ge 0,\, x,y \in \T,
$$
and therefore,
$$
v^\xi(t,x) - v^\xi(t,y) 
\le \sup_{\alpha \in \cA_t} 
\E \int_t^\infty e^{-\beta(u-t)}\,\|\xi\|_{t,\infty}\, |x-y| \ \d u
=\frac{\|\xi\|_{t,\infty}}{\beta}\, |x-y|.
$$
\end{proof}

The following estimate follows directly from the Ito's formula. 

\begin{lemma}
\label{lem:os}
For any $\xi \in \cC, \ n \ge 1$, and $ 0 \le t \le s$,
$$
\left|\E[\cos(n X^\xi_s)]\right|+
\left|\E[\sin(nX^\xi_s)]\right| \le 
2 e^{-\tfrac{n^2 \sigma^2}{2}(s-t)}
+ \frac{4  \|\xi\|_{t,\infty}}{n \beta\, \sigma^2}.
$$
\end{lemma}
\begin{proof}
Set $A_t:= \E[\cos(n X^\xi_t)]$.
Ito formula implies that
$$
A_s-A_t
= n\, \E[\int_t^s v^\xi_x(u,X^\xi_u)\sin(nX^\xi_u)\ \d u]
- \frac{n^2\sigma^2}{2} \int_t^s  A_u\ \d u.
$$
By the previous lemma,
\begin{align*}
A_s& \le e^{-\tfrac{n^2\sigma^2}{2}(s-t)}\ A_t
+\int_t^s  \frac{n \|\xi\|_{t,\infty}}{ \beta}\ 
\ \exp(-\frac{n^2 \sigma^2}{2}(u-t))\d u\\
&=   e^{-\tfrac{n^2\sigma^2}{2}(s-t)}
+ \frac{2  \|\xi\|_{t,\infty}}{n \beta\, \sigma^2}\,
[1- e^{-\tfrac{n^2\sigma^2}{2}(s-t)}].
\end{align*}
The inequality for the sin is proved exactly the same way.
\end{proof}

\subsection{Proof of Lemma~\ref{lem:incoherent}}
\label{ss.proof}

Let $\bmu=(\mu_t)_{t \ge 0}$ be
a solution of the Kuramoto mean field game.
By Lemma~\ref{lem:fixed}, $\xi_t:=\xi_t(\bmu)$ given by \eqref{eq:xi}
is a fixed-point of $\kappa \cT(\cdot;\mu_0)$. Hence,
$\xi_t=\kappa\, \left(\E[\cos(X^\xi_t)], \E[\sin(X^\xi_t)]\right)$.
By Lemma~\ref{lem:os}, for any $0\le t\le \tau$,
\begin{align*}
\|\xi\|_{\tau,\infty}
&= \kappa\ \sup_{s \ge \tau} [\left|\E[\cos(X^\xi_s)\right| +\left|\E[\sin(X^\xi_s)\right|]
 \le  \kappa\ \sup_{s \ge \tau}\ [  2 e^{-\tfrac{\sigma^2}{2}(s-t)}
+ \frac{4 \, \|\xi\|_{t,\infty}}{\beta\, \sigma^2}]\\
&=  2 \kappa e^{-\tfrac{\sigma^2}{2}(\tau-t)}
+ \frac{4  \kappa\ \|\xi\|_{t,\infty}}{\beta\, \sigma^2}.
\end{align*}
As $\|\xi\|_{\tau,\infty}$ is non-increasing in $\tau$,
it has a limit, and the above inequality implies that
$$
\lim_{\tau \to \infty} \, \|\xi\|_{\tau,\infty}
\le \lim_{\tau \to \infty} \,
2  \kappa e^{-\tfrac{ \sigma^2}{2}(\tau-t)}
+ \frac{4  \kappa\ \|\xi\|_{t,\infty}}{\beta\, \sigma^2}
=\frac{4  \kappa\ \|\xi\|_{t,\infty}}{\beta\, \sigma^2},
\qquad \forall\, t \ge 0.
$$
We now take the limit as $t$ tends to infinity.
The result is the following,
$$
\lim_{\tau \to \infty} \, \|\xi\|_{\tau,\infty}
\le 
\frac{4  \kappa}{\beta\, \sigma^2} \lim_{t \to \infty} \, 
\|\xi\|_{t,\infty}.
$$
Thus for $\kappa < \beta \sigma^2/4$,
$\lim_{t \to \infty} \, 
\|\xi\|_{t,\infty} = 0$.
By Lemma~\eqref{lem:os},
$$
 \lim_{t \to \infty} \, 
 \left|\E[\cos(n X^\xi_t)]\right|+
\left|\E[\sin(nX^\xi_t)]\right| =0,
$$
for every $n$.  Let $f$ be a twice continuously
differentiable function on $\T$.
Then,
$$
f(x) = c_0 + \sum_{n=1}^\infty [c_n \cos(nx) +e_n \sin(nx) ],
$$
for some constants $c_n,e_n$.  One may directly
show that they satisfy $\sum_{n=1}^\infty [|c_n| +|e_n| ]<\infty$. 
Hence, by dominated convergence,
$$
 \lim_{t \to \infty} \,  \mu_t(f) = c_0 +  \lim_{t \to \infty} \, 
 \sum_{n=1}^\infty (c_n \E[\cos(n X^\xi_t) +e_n \E[\sin(n X^\xi_t)])
 = c_0 = U(f).
$$
This implies the 
convergence of $\mu_t$ to the uniform
distribution $U$.
\qed

\section{Sub-critical case: \texorpdfstring{$\kappa < \kappa_c$}{K < Kc}}
\label{sec:sub}

For a positive constant $\lambda >0$,
consider the subspace of $\cC$ given by,
$$
\cC_\lambda:= \{ \ \xi =(\gamma,\eta) \in \cC\ :\  
\|\xi\|_\lambda <\infty \},
\quad
\text{where}
\quad
\|\xi\|_{\lambda} := \sup_{t\ge 0}\ e^{\lambda t}|\xi_t|=
\sup_{t \geq 0} e^{\lambda t} \left( |\gamma(t)| + |\eta(t)| \right).
$$
Then, $(\cC_\lambda , \| \cdot\|_{\lambda})$ is a Banach space.
\subsection{Preliminearies}
\label{ss:pre}

Let  $v^\xi(t,x)$ 
be as in  subsection~\ref{ss.spaces},
and recall that 
the optimal state processes $X^\xi$ and $X^{\xi,(t,x)}$
solve the  same stochastic differential
equation \eqref{eq:optsde} but with different  
initial conditions.  Namely,
$X^\xi_0=X_0$  satisfies $\cL(X_0)=\mu_0$, 
and $X^{\xi,(t,x)}_t=x$.
The drift term in the stochastic differential equation
\eqref{eq:optsde} is $v^\xi_x$. 
We differentiate the equation \eqref{eq:dpe},
to show that it solves the following equation,
$$
\beta v^\xi_x
-v^\xi_t(t,x) - \cM(v^\xi_x)(t,x) 
=(\ell_\xi)_x(t,x)
= \gamma(t) \sin(x)-\eta(t) \cos(x),
$$
where $\cM$ is the infinitesimal
generator of the stochastic 
differential equation \eqref{eq:optsde}, 
i.e., for a smooth function
$\varphi$ of $(t,x)$,
$$
\cM(\varphi)(t,x) = 
 \frac{\sigma^2}{2}\ \varphi_{xx}(t,x)
- \varphi_x(t,x)\, v^\xi_x(t,x).
$$
Hence, we have the following representation
of $v^\xi_x$,
\begin{equation}
\label{eq:vxix}
v^\xi_x(t,x)= \int_t^\infty\, e^{-\beta(u-t)}\, \left[
\gamma(u) B^\xi_u(t,x) - \eta(u) A^\xi_u(t,x) \right]\, \d u,
\end{equation}
where
$$
A^\xi_u(t,x;\xi) : = \E\left[\cos(X^{\xi,(t,x)}_u)\right],
\qquad
B^\xi_u(t,x;\xi) : = \E\left[\sin(X^{\xi,(t,x)}_u)\right].
$$
\vspace{5pt}

\begin{lemma}
\label{lem:vx1}
For every $t \ge 0, x \in \T$,
$$
\|v^\xi_x\|_\lambda^*
:= \sup_{t \ge 0, x\in \T}
|v^\xi_x(t,x)| e^{ \lambda t}\, \le \frac{\|\xi\|_\lambda}{\beta}.
$$
\end{lemma}

\begin{proof}
We use the representation \eqref{eq:vxix} with the 
estimate 
$$
|\gamma(u) B^\xi_u(t,x) - \eta(u) A^\xi_u(t,x)| 
\le |\gamma(u)| + |\eta(u)| \leq \| \xi \|_\lambda e^{-\lambda u}
\le \|\xi\|_\lambda e^{-\lambda t},\qquad \forall\, u \ge t,
$$
to obtain
$$
|v^\xi_x(t,x)| \le \int_t^\infty e^{-\beta (u-t)}\,  \|\xi\|_\lambda e^{-\lambda t} \d u
=  \frac{\|\xi\|_\lambda}{\beta}\, e^{-\lambda t}.
$$
\end{proof}
We close this subsection with a simple application
of the Ito's rule.

\begin{lemma}
\label{lem:sde1} Let $X$ be a solution of
the stochastic differential equation \eqref{eq:optsde}
on $(t,\infty)$.  Then, for $f(x)=\cos(x)$ or $f(x) = \sin(x)$,
$$
\left|\E[ f(X_u) ]- e^{-\tfrac{\sigma^2}{2} (u-t)}\, \E[ f(X_t) ] \right|
\le \frac{2\, \|v^\xi_x\|^*_\lambda}{\sigma^2}\,  e^{- \lambda t },
\qquad \forall \, u \ge t.
$$
\end{lemma}

\begin{proof}
We only consider $f(x)=\cos(x)$, 
the other case is proved in exactly the same way. 
By Ito's rule,
$$
\d \cos(X_u) = -  \sin(X_u)\, \d X_u - \frac{\sigma^2}{2} \, \cos(X_u)\, \d u.
$$
Therefore,
$$
\E[ \cos(X_u) ] = e^{-\tfrac{\sigma^2}{2} (u-t)}\E[ \cos(X_t) ] 
+ \int_t^u \, e^{-\tfrac{\sigma^2}{2} (u-s)}\, 
\E[v^\xi_x(s,X_s)\, \sin(X_s)]\, \d s.
$$
Moreover, 
$\E[v^\xi_x(s,X_s)\, \sin(X_s)] \le \|v^\xi_x\|^*_\lambda\,  e^{- \lambda t}$
for every $s \ge t$.  We substitute this into
the above equation to complete the proof of the
claimed inequality.

\end{proof}

Set
$$
\rho:=\frac{\kappa_c}{\sigma^2} = \beta+\frac{\sigma^2}{2}.
$$

\begin{lemma}
\label{lem:rep}
For $\xi =(\gamma, \eta)\in \cC_\lambda$,
the negative drift $v^\xi_x$ has the  representation,
$$
v^\xi_x(t,x)= w^\xi(t,x)+ r^\xi(t,x)
= M_t(\xi) \sin(x) -N_t(\xi) \cos(x)+ r^\xi(t,x),
$$
where
$$
M_t(\xi)= \int_t^\infty e^{- \rho(s-t)}\,  \gamma(s)\, \d s,
\qquad
N_t(\xi)= \int_t^\infty e^{- \rho(s-t)}\,  \eta(s)\, \d s.
$$

\begin{equation}
\label{eq:MNr}
\|(M_t(\xi), N_t(\xi))\|_\lambda \, \le \frac{\|\xi\|_\lambda}{\rho}\,
\qquad
| r^\xi(t,x)| \le \frac{4}{\beta^2\sigma^2}\, \|\xi\|_\lambda^2\, e^{-2\lambda t}.
\end{equation}
\end{lemma}
\begin{proof} We first use Lemma~\ref{lem:sde1} with
$X^{\xi,(t,x)}$ to obtain
\begin{align*}
A_u^\xi(t,x) = e^{-\tfrac{\sigma^2}{2}(u-t)}\, \cos(x) + a^\xi_u(t,x),
\quad
\text{and}
\quad
|a^\xi_u(t,x)| \le \frac{2\, \|v^\xi_x\|^*_\lambda}{\sigma^2}\,  e^{- \lambda t},\\
B_u^\xi(t,x) = e^{-\tfrac{\sigma^2}{2}(u-t)}\, \sin(x) + b^\xi_u(t,x),
\quad
\text{and}
\quad
|b^\xi_u(t,x)| \le \frac{2\, \|v^\xi_x\|^*_\lambda}{\sigma^2}\,  e^{- \lambda t}.
\end{align*}
Let $w^\xi, M_t, N_t$ be as in the statement of the lemma. Then, formula \eqref{eq:vxix} gives,
$$
r^\xi(t,x):= v^\xi_x(t,x) -w^\xi(t,x) =
\int_t^\infty e^{-\beta(u-t)}\, 
\left[ \gamma(u) b^\xi_u(t,x) - \eta(u) a^\xi_u(t,x) \right]\, \d u.
$$
The claimed estimates of $N,M,r^\xi$ now follows
directly from Lemma~\ref{lem:vx1}.

\end{proof}

\subsection{Linearization}
\label{ss:linear}

For small $\xi$, one expects the
value function and its derivatives to be small.
We exploit this formal observation and obtain the following
representation of the map $\cT(\cdot;\mu_0)$.
Recall, $w^\xi, N(\xi),M(\xi)$ of Lemma~\ref{lem:rep},
and $d(\mu_0)$ of \eqref{eq:d}.

\begin{proposition}
\label{pro:linear} For any $\xi \in \cC_\lambda$, $t \ge 0$,
$\mu_0 \in \cP(\T)$, and $\lambda \le \sigma^2/8$,
$$
\cT_t(\xi;\mu_0) = e^{-\tfrac{\sigma^2}{2} t}\, 
(\mu_0(\cos),\mu_0(\sin)) + 
\Xi_t(\xi)+ \,   R_t(\xi,\mu_0),
$$
where the linear operator $\Xi$ is given by
$$
\Xi_t(\xi) := 
\frac12 \int_0^t e^{-\tfrac{\sigma^2}{2} (t-u)}\,
(M_u(\xi)\, , \, N_u(\xi))\, \d u ).
$$
Moreover, there is
$c(\beta,\sigma)>0$ satisfying
\begin{equation}
\label{eq:estR}
|R_t(\xi,\mu_0)| \le \, 
[\, \frac{8d(\mu_0)}{\kappa_c}\,
 \|\xi\|_\lambda 
+  c_2(\beta,\sigma)\, \|\xi\|_\lambda^2\, ]
\, e^{-2 \lambda t}.
\end{equation}
\end{proposition}
\begin{proof}
Recall that for  $\xi=(\gamma,\eta)\in \cC_\lambda$
and $t \ge 0$,
$$
\cT_t(\xi;\mu_0)           =
(\E[\cos(X^\xi_t), \E[\sin(X^\xi_t)]),
\qquad t \ge0,
$$
where $X^\xi$ solves \eqref{eq:optsde} 
with initial condition $\cL(X_0)=\mu_0$.
We continue in several steps.
\vspace{5pt}

\noindent
\emph{Step 1.}
By Ito's rule,
\begin{equation}
\label{eq:gamma}
\E[\cos(X^\xi_t)] = e^{-\tfrac{\sigma^2}{2} t}\, \mu_0(\cos)
+ \int_0^t  e^{-\tfrac{\sigma^2}{2} (t-u)}\,
\E[v^\xi_x(u,X^\xi_u) \sin(X^\xi_u)]\, \d u.
\end{equation}
Moreover, by Lemma~\ref{lem:rep},
$$
v^\xi_x(u,X^\xi_u) \sin(X^\xi_u)
= M_u(\xi) \sin^2(X^\xi_u) -N_u(\xi) \sin(X^\xi_u) \cos(X^\xi_u)
+r^\xi(u,X^\xi_u)\, \sin(X^\xi_u).
$$
\emph{Step 2.} Set $Y_u:= \sin^2(X^\xi_u)-\tfrac12$.
By Ito's rule,
$$
\d Y_u = 2 v^\xi_x(u,X^\xi_u) \sin(X^\xi_u)\, \cos(X^\xi_u)\, \d u
- 2\sigma^2 Y_u \d u + (\ldots) \d B_u.
$$ 
This implies that
$$
\E[Y_u] = e^{-2\sigma^2 u} \E[Y_0]  +
\int_0^u 2 e^{-2\sigma^2 (u-s)}
\E[v^\xi_x(s,X^\xi_s) \sin(X^\xi_s)\, \cos(X^\xi_s)]\, \d s.
$$
Then, by Lemma~\ref{lem:vx1}
and \eqref{eq:d},
\begin{align}
| \E[\sin^2(X^\xi_u)]-\frac12|  & \le  
e^{-2\sigma^2 u}\,  | \E[\sin^2(X^\xi_0)]-\frac12| 
+ \frac{\|\xi\|_\lambda}{ \beta\sigma^2} \, 
e^{-\lambda u} \nonumber \\
& \le  e^{-2\sigma^2 u}\, d(\mu_0) 
+ \frac{\|\xi\|_\lambda}{ \beta\sigma^2} \, 
e^{-\lambda u}. \label{eq:estimate sin2 - 1/2}
\end{align}
A similar calculation implies that
\begin{equation}
| \E[\sin(X^\xi_u)\cos(X^\xi_u)]| \le  e^{-2\sigma^2 u}\, d(\mu_0) 
+ \frac{\|\xi\|_\lambda}{ \beta\sigma^2} \, 
e^{-\lambda u}.
\label{eq:estimate sincos}
\end{equation}

\noindent
\emph{Step 3.}  Set
$$
R^1_u(\xi;\mu_0):= \E[v^\xi_x(u,X^\xi_u) \sin(X^\xi_u)] -
\frac12 \, M_u(\xi).
$$ 
We directly estimate $R^1$ by using the 
previous steps and \eqref{eq:MNr}. The result is the following,
\begin{align*}
|R^1_u(\xi;\mu_0)|&\le 
(|N_u(\xi)| +|M_u(\xi)|) \, [ e^{-2\sigma^2u} d_0(\mu_0)
+ \frac{\|\xi\|_\lambda}{\beta\sigma^2} e^{-\lambda u} ]
+ \E[|r^\xi(u,X^\xi_u)|]\\
&\le  \frac{2 \|\xi\|_\lambda}{\rho}\, e^{-\lambda u}\,
[ e^{-2\sigma^2 u} d_0(\mu_0)
+ \frac{\|\xi\|_\lambda}{\beta\sigma^2}  \, e^{-\lambda u}]
+  \frac{4}{\beta^2\sigma^2}\, \|\xi\|_\lambda^2 \, e^{-2\lambda u}\\
&\le [\, \frac{2 d(\mu_0)}{\rho}\, \|\xi\|_\lambda 
+  c(\beta,\sigma)\, \|\xi\|_\lambda^2\, ]\, e^{-2 \lambda u},
\end{align*}
where
$$
c(\beta,\sigma):=
\frac{(2 \beta+4 \rho)}{\rho \beta^2\sigma^2}=
\frac{(12\beta+4  \sigma^2)}{(2\beta+\sigma^2) \beta^2\sigma^2}.
$$
\vspace{5pt}

\noindent
\emph{Step 4.} Using Step 3 in \eqref{eq:gamma}, we obtain
$$
\E[\cos(X^\xi_t)] = e^{-\tfrac{\sigma^2}{2} t}\, \mu_0(\cos)
+\frac12  \int_0^t  e^{-\tfrac{\sigma^2}{2} (t-u)}\, M_u(\xi)\, \d u
+R^2_t(\xi,\mu_0),
$$
where
$$
R^2_t(\xi,\mu_0) =
\int_0^t  e^{-\tfrac{\sigma^2}{2} (t-u)}\, 
R^1_u(\xi,\mu_0)\, \d u.
$$
By the estimate obtained in Step 3,
and as $\lambda < \sigma^2/8$,
\begin{align*}
|R^2_t(\xi,\mu_0) | & \le 
\int_0^t  e^{-\tfrac{\sigma^2}{2} (t-u)}\, 
[ \frac{2 d(\mu_0)}{\rho}\, \|\xi\|_\lambda 
+  c(\beta,\sigma)\, \|\xi\|_\lambda^2]\, e^{-2 \lambda u}\d u\\
&\le  \frac{4}{\sigma^2}\,
[ \frac{2 d_0(\mu_0)}{\rho}\, \|\xi\|_\lambda 
+  c(\beta,\sigma)\, \|\xi\|_\lambda^2]\,
e^{-2 \lambda t}.
\end{align*}

\noindent
\emph{Step 5.} Proceeding exactly as in the
previous steps, we also obtain,
$$
\E[\sin(X^\xi_t)] = e^{-\tfrac{\sigma^2}{2} t}\, \mu_0(\sin)
+\frac12  \int_0^t  e^{-\tfrac{\sigma^2}{2} (t-u)}\, N_u(\xi)\, \d u
+R^3_t(\xi,\mu_0),
$$
where
$$
|R^3_t(\xi,\mu_0) |  \le  \frac{4}{\sigma^2}\,
[ \frac{2 d(\mu_0)}{\rho}\, \|\xi\|_\lambda 
+  c(\beta,\sigma)\, \|\xi\|_\lambda^2]\,
e^{-2 \lambda t}.
$$
As $\rho \sigma^2=\kappa_c$,
above estimates implies \eqref{eq:estR}
with $c_2(\beta,\sigma)=
4 c(\beta,\sigma)/\sigma^2$.
\end{proof}

\begin{corollary}
\label{cor:MN}
For every $\lambda$ small,
$\Xi$ is bounded on $\cC_\lambda$.
In particular, 
for all $\xi \in \cC_\lambda$,
$\cT(\xi;\mu_0) \in \cC_\lambda$
and 
$$
\frac{\|\Xi(\xi)\|_\lambda}{\|\xi\|_\lambda} \le 
\frac{1}{  (\kappa_c-2\lambda\rho)},
\qquad \forall\, \xi \in \cC_\lambda.
$$
\end{corollary}
\begin{proof}
By the definitions of $M$, and $N$,
$$
|(M_u,N_u)|
\le \int_u^\infty e^{-\rho (s-u)} \, e^{-\lambda u} \|\xi\|_\lambda \d u
=\frac{\|\xi\|_\lambda}{\rho +\lambda}\, e^{-\lambda u}
\le \frac{\|\xi\|_\lambda}{\rho}\, e^{-\lambda u}.
$$
Then,
$$
|\Xi_t(\xi)| \le \frac12 \int_0^t e^{-\tfrac{\sigma^2}{2} (t-u)}\,
\frac{\|\xi\|_\lambda}{\rho}\, e^{-\lambda u}\, \d u\\
= \frac{\|\xi\|_\lambda}{(\sigma^2-2\lambda) \rho}\, e^{-\lambda t}
= \frac{\|\xi\|_\lambda}{ (\kappa_c-2\lambda\rho)}\, e^{-\lambda t},
$$
 where in the last calculation we used the identity $\kappa_c= \sigma^2 \rho$.
 As all terms in the representation of $\cT(\xi;\mu_0)$ are in 
$\cC_\lambda$, consequently, so is $\cT(\xi;\mu_0)$.

\end{proof}

Note that for all $\kappa <\kappa_c$,
and all sufficiently small $\lambda$, the map 
$\xi \in \cC_\lambda \mapsto \kappa \Xi(\xi) \in \cC_\lambda$
is a contraction.  Therefore, 
in view of Proposition~\ref{pro:linear},
$\kappa \cT$ is
equal to a contraction perturbed by a quadratic
nonlinearity.  This formal observation drives
the subsequent analysis.

\subsection{Proof of Theorem~\ref{th:main}}
\label{ss:proofmain}

We start with a uniform bound.
Recall $d(\mu_0)$ of \eqref{eq:d}.

\begin{lemma}
\label{lem:uniform}
For every $\kappa<\kappa_c$,
there are $C_\kappa,c_\kappa, {\lambda_\kappa}>0$ depending
only on  $\kappa, \beta, \sigma$, such that
if $d(\mu_0) \le c_\kappa$, then
$$
\|\xi\|_{\lambda_\kappa} \le C_\kappa \quad
\Rightarrow
\quad
\|\kappa\, \cT(\xi;\mu_0)\|_{\lambda_\kappa} \le C_\kappa.
$$
\end{lemma}
\begin{proof}

We fix $\xi \in \cC_\lambda$,
$\mu_0$, $\kappa<\kappa_c$ and set  
$$
\delta:= \frac{\kappa_c-\kappa}{4 \kappa_c},
\qquad \Rightarrow\qquad \kappa =\kappa_c (1-4 \delta).
$$
Choose ${\lambda_\kappa} <\delta \sigma^2/2$.  
Then,
$$
(\kappa_c-2{\lambda_\kappa}\rho)(1-3\delta)
= \kappa_c(1-\frac{2 \lambda_\kappa}{\sigma^2})
(1-3\delta)
\ge \kappa_c(1-\delta)
(1-3\delta)\ge \kappa_c(1-4\delta) = \kappa.
$$ 
By definition 
$\kappa\, \cT(\xi;\mu_0)_0=\kappa (\mu_0(\cos),\mu_0(\sin))$.
Therefore, $|\kappa\,  \cT(\xi;\mu_0)_0| \le 2 \kappa d(\mu_0)$. Then, 
by Proposition~\ref{pro:linear} and Corollary~\ref{cor:MN},
\begin{align*}
\|\kappa\,  \cT(\xi;\mu_0)\|_{\lambda_\kappa} 
&\le |\kappa\,  \cT(\xi;\mu_0)_0| + [ \frac{\kappa}{\kappa_c-2{\lambda_\kappa}\rho}
+\frac{8 \kappa d (\mu_0)}{\kappa_c}]\, \|\xi\|_{\lambda_\kappa}
+\kappa  c_2(\beta,\sigma)\, \|\xi\|_{\lambda_\kappa}^2\\
&\le  2 \kappa d(\mu_0) + [ (1-3\delta)
+\frac{8 \kappa d (\mu_0)}{\kappa_c}]\, \|\xi\|_{\lambda_\kappa}
+\kappa  c_2(\beta,\sigma)\, \|\xi\|_{\lambda_\kappa}^2.
\end{align*}
Set 
$$
C_\kappa=C(\kappa,\beta,\sigma):= \frac{\delta}{\kappa c_2(\beta,\sigma)},
\qquad
c_\kappa=c(\kappa,\beta,\sigma):=\min\{ \frac{\delta C_\kappa}{2\kappa}\ ,\ 
\frac{\delta}{8}\}.
$$
Then, if $d(\mu_0) \le c_\kappa$ and $\|\xi\|_\lambda \le C_\kappa$,
\begin{align*}
\|\kappa\,  \cT(\xi;\mu_0)\|_{\lambda_\kappa} &\le
2\kappa c_\kappa + [(1-3\delta) + 8 d(\mu_0)] \|\xi\|_\lambda 
+[\kappa c_2(\beta,\sigma)  \|\xi\|_\lambda]\ \|\xi\|_\lambda\\
&\le \delta C_\kappa + (1-2\delta) C_\kappa 
+ [\kappa c_2(\beta,\sigma) C_\kappa]\ C_\kappa\\
&= C_\kappa \ [\delta + (1-2\delta) + \delta]= C_\kappa.
\end{align*}
This completes the proof of this lemma.
\end{proof}

Let ${\lambda_\kappa}, C_\kappa, c_\kappa>0$
be as in the above lemma and  set
$$
\cB_\kappa := \{ \, \xi \in \cC_{\lambda_\kappa}\, 
:\, \|\xi\|_{\lambda_\kappa} \le C_\kappa\, \}
\subset \cC_\lambda,
\qquad \forall\, \lambda \in (0,{\lambda_\kappa}].
$$
We have shown above that $\kappa\, \cT(\cdot;\mu_0)$ maps
$\cB_\kappa$ into itself provided that $d(\mu_0)\le c_\kappa$.

\begin{lemma}
\label{lem:compact}
For every $0<\lambda <{\lambda_\kappa}$
and $d(\mu_0)\le c_\kappa$,
$\kappa\, \cT(\cB_\kappa;\mu_0)$ is pre-compact
in $\cC_\lambda$. 
\end{lemma}
\begin{proof}
Fix $0<\lambda <{\lambda_\kappa}$
and $\mu_0$ with $d(\mu_0)\le c_\kappa$.
Let $\xi^n$ be a sequence $\cB_\kappa$,
and set $\zeta^n:= \kappa\, \cT(\xi^n;\mu_0)$.
By the previous lemma, $\zeta^n \in \cB_\kappa$.
As in the proof of Lemma~\ref{lem:existence}
given in Appendix~\ref{ss:exist}, we use Lemma~\ref{lem:Alip}
with  Arzel\`a–Ascoli in a diagonal argument
to construct a subsequence, denoted by $n$ again, 
and $\zeta^*\in \cB_\kappa$ such that
$\zeta^n$ converges to $\zeta^*$ uniformly
on every compact set $[0,T]$.
Since $\zeta^n, \zeta^* \in \cB_\kappa$,
for every $n, T$,
$$ 
\sup_{t \ge T} |\zeta^n_t-\zeta^*_t| e^{\lambda t}
\le \sup_{t \ge T} |\zeta^n_t-\zeta^*_t| e^{(\lambda-\lambda_\kappa)t} e^{\lambda_\kappa t}
\le \|\zeta^n-\zeta^*\|_{\lambda_\kappa} e^{-({\lambda_\kappa} -\lambda)T} 
\le 2C_\kappa e^{-({\lambda_\kappa} -\lambda)T} .
$$
Given $\eps>0$, we choose $T_\eps>0$  such  that 
$2C_\kappa e^{-({\lambda_\kappa} -\lambda)T_\eps} \le \eps$.
As $\zeta^n$ converges to $\zeta^*$
uniformly on every bounded set,
there exists $n_\eps$ satisfying
$$
\sup_{t\in [0,T_\eps]}  |\zeta^n_t-\zeta^*_t| 
e^{\lambda t} \le \eps,
\quad \forall\, n \ge n_\eps.
$$
Hence, for every $n \ge n_\eps$, $\|\zeta^n-\zeta^*\|_\lambda \le \eps$.
So we conclude that $\zeta^n$ converges to $\zeta^*$
in $\|\cdot\|_\lambda$, proving that 
$\kappa\, \cT(\cB_\kappa; \mu_0)$ is pre-compact
in $\cC_\lambda$.
\end{proof}
\vspace{5pt}

\noindent
\emph{Proof of Theorem~\ref{th:main}}.
Let $\lambda_\kappa, c_\kappa, C_\kappa$
be as in the  Lemma~\ref{lem:uniform}.  
Fix $\lambda \in (0,\lambda_\kappa)$.
Suppose that 
the initial distribution $\mu_0$ satisfies
$d(\mu_0)\le c_\kappa$.
We have shown that $\kappa\, \cT(\cdot;\mu_0)$
is pre-compact on the convex set
$\cB_\kappa$ and it  maps $\cB_\kappa$  onto itself.
The continuity of  $\kappa\, \cT(\cdot;\mu_0)$ can be proved 
as in Lemma~\ref{lem:Tcont}.
Therefore, we can use the Schauder fixed point theorem
to conclude that there exists $\xi^* \in \cC_\lambda$
so that $\xi^* =\kappa\, \cT(\xi^*;\mu_0)$.
Let $X^*$ be the optimal process 
for the problem \eqref{eq:scp1}
with $\xi^*$.  
In view of Lemma~\ref{lem:fixed},
$(\cL(X^*_t))_{t \ge 0}$
is a solution of the Kuramoto
mean field game with interaction parameter $\kappa$
and initial distribution $\mu_0$.
We now use Lemma~\eqref{lem:os} as in 
subsection~\ref{ss.proof}, to conclude that
$\cL(X^*_t)$ converges to the uniform distribution. 

As $\xi^* \in \cB_\kappa$, we have $\|\xi^*\|_\lambda \leq C_\kappa$. Therefore, by
\eqref{eq:estimate sin2 - 1/2} and 
\eqref{eq:estimate sincos}, we 
conclude that \eqref{eq:exponential rate 
convergence th} holds with 
$\lambda^*_\kappa = \lambda$. 

\qed

\appendix

\section{Existence of solutions}
\label{app:exist}

We first approximate the
infinite horizon problem \eqref{eq:scp1}
by finite horizon problems and 
prove the existence 
of solutions for them.  Then,
we use a limiting argument to 
construct solutions to the original problem.

\subsection{Finite horizon problem}
\label{ss.fhp}
For a finite horizon $T$, 
we modify the control problem
\eqref{eq:scp1} slightly and consider
$$
\text{minimize } \alpha \ \mapsto
\ J_{\bmu}(\alpha;T)
:= \E \int_t^T e^{-\beta(u-t)}
[\ell_\mu(u,X^{\alpha}_u) +\tfrac12 \alpha_u^2] \
\d u,
\quad 
t \in [0,T], x \in \T,
$$
where $\ell_\mu$
and $X$ are as in \eqref{eq:scp1}. The solution
to the finite horizon Kuramoto mean field game
is defined exactly as in
Definition~\ref{def:KMFG}.  

Arguments  
of the subsection~\ref{ss.spaces}
leading to Lemma~\ref{lem:fixed} can be
followed \emph{mutatis mutandis} to obtain a
similar fixed point characterization of the
solutions.  Indeed, let
$$
\cC_{T,\kappa}:= \{ \ \xi =(\gamma,\eta) : [0,T]
\mapsto \R^2\ :\ 
\text{continuous and}\ \|\xi\|_\infty \le
\kappa\  \},
$$
and for $t \in [0,T]$, $x \in \T$ set
\begin{equation}
\label{eq:vtx}
v^{\xi,T}(t,x):= \inf_{\alpha \in \cA_t}
J^{\xi,T}(t,x,\alpha)
:= \E \int_t^T e^{-\beta(u-t)}
[\ell_\xi(u,X^{\alpha,(t,x)}_u) +\tfrac12
\alpha_u^2] \ \d u,
\end{equation}
where $\ell_\xi$ is as in \eqref{eq:ell} and
$X^{\alpha,(t,x)}$ is as in \eqref{eq:vx}.  

Note that the corresponding dynamic programming
equation is exactly 
\eqref{eq:dpe}.  The only
difference is that the equation holds for 
$ t \in (0,T)$ 
and $v^{\xi,T}$ satisfies the terminal
condition $v^{\xi,T}(\cdot,T)\equiv 0$.  
This equation has a smooth solution, and 
in particular, the Lipschitz estimate
$|v^{\xi,T}_x(x,t)| \le \|\xi\|_\infty/\beta$ is
proved as in 
the proof of Lemma~\ref{lem:vx}.
Also the optimal state process $X^{\xi,T}$
starting from
any initial condition $X_0$ is the unique
solution of the stochastic differential
equation \eqref{eq:optsde} with $ v^{\xi,T}_x$
replacing $v^\xi_x$, i.e.,
\begin{equation}
\label{eq:optsde1}
\d X^{\xi,T}_t = -\, v^{\xi,T}_x(u,X^{\xi,T}_u)\ \d u + \sigma \d B_t,
\end{equation}

\begin{definition}\label{def:FKMFG}
{\rm{A flow of probability measures
${\bmu}:= (\mu_t)_{t \in [0,T]}$ is a solution
of the}} finite horizon Kuramoto mean field game
with initial data $\mu_0$, 
{\rm{if and only 
the solution of \eqref{eq:optsde} with $\cL(X_0)=\mu_0$ satisfies
$\cL(X_t)=\mu_t$ for all $t \in [0,T]$.}}
\end{definition}
As in Lemma~\ref{lem:fixed}
to prove the existence of a solution to the
finite horizon 
Kuramoto mean field game, it suffices 
construct a fixed point of
$\kappa \cT(\cdot;T,\mu_0)$, where
$$
\cT (\cdot;T,\mu_0) : \xi \in \cC_{T,\kappa} 
\ \mapsto \ 
\cT(\xi;T,\mu_0):=(\E[\cos(X^{\xi,T}_t)],
\E[\sin(X^{\xi,T}_t)])_{t \in [0,T]}.
$$

\subsection{ A convergence result}
\label{ss.convergence}

In this subsection,
we consider the 
stochastic optimal control problem \eqref{eq:vtx} 
with a running cost $\ell_\xi$ 
given by \eqref{eq:ell} and with both finite and
infinite $T$. 
It is classical that
the value function $v^{T,\xi}$ of \eqref{eq:vtx} or
$v^{\infty,\xi}:=v^\xi$ of \eqref{eq:vx}
are smooth, classical solutions of the
dynamic programming equation \eqref{eq:dpe}.
The main result of this subsection is the following 
convergence result that is used repeatedly in our forthcoming arguments.

\begin{lemma}
\label{lem:converge}
For $T \le \infty$,
suppose that $T_n $ converges to $T$
and a sequence $\xi^n \in \cC_{T_n,\kappa}$ converges 
locally uniformly to $\xi^* \in \cC_{T,\kappa}$.
Then,
$v^n:= v^{\xi^n,T_n}$ and $v^n_x$ converge
locally uniformly to $v^*:= v^{T,\xi^*}$
and to $v^*_x$, respectively.
\end{lemma}
\begin{proof}
Convergence  of the value function follows directly from the definitions.
Also we have argued earlier that $v^{\xi,T}_x(t,x)| \le \|\xi\|_\infty/\beta$.
In particular, $v^n_x$ is uniformly bounded.

Set $w^n:= v^n_{xx}$ 
and $\ell^n := \ell_{\xi^n}$.  The dynamic programming
equation \eqref{eq:dpe} implies that $w^n$ satisfies the 
linear parabolic equation
$$
-w^{n}_t(t,x) +\beta w^n(t,x)
- \frac{\sigma^2}{2} w^{n}_{xx}(t,x) + w^{n}_x(t,x)\,
v^{n}_x(t,x) = (\ell^n)_{xx}(t,x) - (w^n(t,x))^2,
$$
with $w^n(T,\cdot)\equiv 0$ when $T<\infty$.
Since $|\xi^n|\le \kappa$, we have
$(\ell^n)_{xx}(t,x) \le \kappa$. Therefore,
Feynman-Kac implies that
$$
w^n(t,x)= v^n_{xx}(t,x) \le \frac{\kappa}{\beta},
\qquad \forall\ t \in [0,T], x \in \T.
$$
Therefore, $\bar{v}^n(t,x):= v^n(t,x) - \frac{\kappa}{2 \beta} x^2$
is concave.  Consider a sequence $(t_n,x_n)$ converging
to $(t_0,x_0)$ and set $p_n := v_x^n(t_n,x_n)  - (\kappa/\beta) x_n$.
Then, $p_n = \bar{v}^n_x(t_n,x_n)$ and since 
$\bar{v}^n$ is concave, we have
$$
\bar{v}^n(t_n,y) \le \bar{v}^n(t_n,x_n) + p_n (y-x_n), \qquad
\forall\, y \in \R.
$$
Since as argued before $v^n_x$ is uniformly
bounded, $p_n$ converges to $p^*$
on a subsequence.
Set $\bar{v}^*(t,x):= v^*(t,x) - \frac{\kappa}{2 \beta} x^2$
and  let $n$ tend to infinity, to conclude that
$$
\bar{v}^*(t_0,y) \le \bar{v}^*(t_0,x_0) + p^* (y-x_0), \qquad
\forall\, y \in \R.
$$
As $\bar{v}^*$ is concave and differentiable, 
the above inequality implies that $p^* = v^*_x(t_0,x_0)$,
proving the local uniform convergence of $v^n_x$
to $v^*_x$.

\end{proof}

\subsection{Finite horizon solution}
\label{ss.fhs}

Fix $T>0$ and for $\xi\in \cC_{T,\kappa}$, let
$X^{\xi,T}$ be as in \eqref{eq:optsde1},
and set
$$
A_t^{\xi,T} := \E[\cos(X^{\xi,T}_t)],\quad
B_t^{\xi,T} := \E[\sin(X^{\xi,T}_t)],\quad
t \in [0,T].
$$
\begin{lemma}
\label{lem:Alip}
There exists a constant $c_*$ depending only
on $\beta, \sigma, \kappa$ such that
$$
|A_s^{\xi,T}-A_t^{\xi,T}|+
|B_s^{\xi,T}-B_t^{\xi,T}| \le c_*(s-t),
\qquad
\forall\ \xi \in \cC_{T,\kappa},\, 0 \le t 
\le s \le T.
$$
\end{lemma}
\begin{proof}
Using Ito's formula as in the  proof of
Lemma~\ref{lem:os}
we arrive at the following estimate:
$$
A_s^{\xi,T} = e^{-\tfrac{\sigma^2}{2}(s-t)}\
A_t^{\xi,T} \ + \int_t^s
e^{-\tfrac{\sigma^2}{2}(u-t)}\,
\E[v^{\xi,T}_x(u,X^{\xi,T}_u) \sin(X^{\xi,T}_u)]
\d u.
$$
Since $v^{\xi,T}_x(x,t)| \le \|\xi\|_\infty/\beta$ 
and $\|\xi\|_\infty \le \kappa$
for every $\xi \in \cC_{T,\kappa}$, we
conclude that $A$ is uniformly Lipschitz.
The statement for $B$ is proved exactly the
same way.
\end{proof}

We have shown that  $\kappa \cT(\cdot;T,\mu_0)$
maps $\cC_{T,\kappa}$ into itself,
and by  Arzel\`a–Ascoli, the above
uniform Lipschitz estimate implies that it is a 
compact map.
\begin{lemma}
\label{lem:Tcont}
For every $T<\infty$, interaction parameter $\kappa$,
and $\mu_0$,
there are solutions to the finite horizon
Kuramoto mean field game.
\end{lemma}
\begin{proof}
We first prove that $\cT(\cdot;T,\mu_0)$ is continuous on
$\cC_{T,\kappa}$.
Suppose that a sequence $\xi^n \in \cC_{T,\kappa}$ 
converges uniformly to
$\xi^*$. Let $v^n:=v^{\xi^n,T}$ be the value
function defined in \eqref{eq:vtx}, and set
$v^*:=v^{\xi^*,T}$.  Since $T<\infty$,
by Lemma~\ref{lem:converge},
$v^n, v^n_x$ converge uniformly to $v^*, v^*_x$.
Let $X^n$ be the solution
of \eqref{eq:optsde} with $v^n_x$
and $X^*$ be the solution with $v_x^*$.
Since $v^n_x$ converges to $v^*_x$ uniformly,
we conclude that $X^n_t$ converges to $X^*_t$
almost surely for every $t \in [0,T]$.
Consequently,
$\E[\cos(X^n_{t}), \sin(X^n_{t})]$ converges to 
$\E[\cos(X^*_{t}), \sin(X^*_{t})]$ for every $t \in [0,T]$, 
and this implies the continuity of $\cT$.

Summarizing, we have  shown that
$\kappa \cT(\cdot;T,\mu_0)$
is a continuous, 
compact operator mapping $\cC_{T,\kappa}$ into itself.
Therefore, we can apply the 
Schauder fixed point theorem to conclude that 
$\kappa \cT(\cdot;T,\mu_0)$ has a fixed point.
Then, the finite horizon
version of Lemma~\ref{lem:fixed} implies that
there are solutions to the finite horizon
Kuramoto mean field game.
\end{proof}

\subsection{Proof of Lemma~\ref{lem:existence}}
\label{ss:exist}

Let $\N$ be the set of all positive integers.  We represent
subsequences by \emph{strictly increasing} functions of $\N$ 
into itself.  Fix $\kappa$ and for $m \in \N$, set 
$$
\cC_m:= \cC(m,\kappa)=\{ \xi :[0,m] \mapsto \R^2\ :\ \|\xi\|_\infty \le \kappa\}.
$$
Let $\bmu^m =(\mu^m_t)_{t \in [0,m]}$ be a solution
of the Kuramoto mean field
game with horizon $m$, and 
$\xi^m:=\xi(\bmu^m) \in \cC_m$ be as in \eqref{eq:xi}.
By Lemma~\ref{lem:Alip} they are uniformly 
Lipschitz continuous,
and by their definition, are bounded by $\kappa$.
We now use the diagonal argument
to construct a locally convergent subsequence.
Set $M_0(n)=n$ for every $n\in \N$.
For $m \in \N$ we recursively construct subsequences
$N_m,M_m:\N \mapsto \N$ as follows.
Suppose that $M_{m}$ is constructed so that $M_m(1) \ge m$,
and  the sequence of
functions 
$$
( \xi^{M_m(n)})_{n\in \N} = 
(\xi^{M_m(1)},\xi^{M_m(2)},\ldots ) \subset \cC_m
$$ 
is uniformly convergent to a function $\xi^{m,*} \in \cC_m$.
If $n \ge 2$, 
$$
M_m(n) \ge M_m(2) \ge M_m(1)+1 \ge m+1.
$$
Hence, $(\xi^{M_m(n)})_{n=2,3,  }$ are all in  $\cC_{m+1}$.
Moreover, they are uniformly Lipschitz continuous on $[0,m+1]$.
By Arzel\`a–Ascoli there exists 
an increasing function $N_m :\N \mapsto \{2,3,\ldots\}$
such that with
$$
M_{m+1}(n):= M_{m}(N_m(n)),\qquad n \in \N,
$$
the sequence of
functions 
$( \xi^{M_{m+1}(n)})_{n\in \N} = 
(\xi^{M_{m+1}(1)},\xi^{M_{m+1}(2)},\ldots ) \subset \cC_{m+1}$ 
is uniformly convergent to a function $\xi^{m+1,*} \in \cC_{m+1}$.
Moreover, 
$M_{m+1}(1)= M_m(N_m(1)) \ge M_m(2) \ge M_m(1)+1\ge m+1$.
Hence, we can repeat the process to construct 
$M_m, N_m$ as claimed.
It is also clear that the limit functions satisfy
the consistency condition 
$$
m \le m'
\qquad \Rightarrow \qquad
\xi^{m,*}_t =\xi^{m',*}_t, \quad \forall \ t \in [0,m].
$$
Then,  the
function $\xi^*_t:=\xi_t^{m,*}$ when $t \in [0,m]$,
is a well-defined and is in $\cC$.  Notice that
by construction,
$\{ M_{m'}(n) : n \in \N \} \ \subset\
\{M_{m}(n):n \in \N \}$, for every $m \le m'$.

Finally, for $n \in \N$, set $K(n):= M_n(n)$.
Then, $(K(n))_{n \ge m} \subset (M_m(n'))_{n'\in \N}$
for every $m \in \N$, i.e.,
$K$ after the index $m$ is a subsequence of $M_m$. 
Therefore, 
$$
\lim_{n \to \infty}  \xi^{K(n)}_t 
= \lim_{n \to \infty}  \xi^{M_m(n)} _t = \xi^{m,*}_t = \xi^*_t,
\qquad \forall\ t \in [0,m].
$$
Moreover, this convergence is uniform.  Hence,
as $n$ tends to infinity
the sequence of functions $\xi^{K(n)}$
converge to $\xi^*$ uniformly on every $[0,m]$.
Set $\ell_m(t,x):= \ell^{\xi^{K(m)}}$, 
$\ell^*:=\ell^{\xi^*}$,
$v^m := v^{\xi^{K(m)},m}$,
and $v^*:=v^{\xi^*}$.  Note that
$v^{\xi^*}$ is also
equal to $v^{\infty,\xi^*}$. 
Then, by Lemma~\ref{lem:converge},
$v^m, v^m_x, \ell_m$ converge locally uniformly to $v^*, v^*_x$,
and respectively $\ell^*$.
As before, let $X^m$ be
given by \eqref{eq:optsde} with $v^m_x$,
and $X^*$ be the solution of \eqref{eq:optsde}
with $v^*_x$. Then, $X^m$ is the optimal
state process for $v^m$ and $X^*$ for $v^*$.
Also, for every $t\ge0$,
$X^m_t$ converges to $X^*_t$ almost surely.
As $\xi^m$ is a fixed point of $\cT(\cdot;T_m,\mu_0)$,
we have
$$
\kappa \ \E[\cos(X^*_t), \sin(X^*_t)]= \lim_{m \to \infty}\,
\kappa\ \E[\cos(X^m_t),\sin(X^m_t)]= 
\lim_{n \to \infty}\, \xi^m_t = \xi^*_t.
$$
Hence, $\xi$ is a fixed point of the map $\kappa \cT(\cdot;\mu_0)$.
By Lemma~\ref{lem:fixed},
the probability flow $(\cL(X^*_t))_{t \ge 0}$
is a solution 
of the Kuramoto mean field game starting
from the distribution $\mu_0$.

\qed

\section{A Comparison Result}
\label{app:compare}

We provide the proof of the comparison result for \eqref{eq:eikonaleq}
which essentially follows from standard techniques.
The fact that the forcing term in the equation
vanishes at the boundary does not allow
us to find an immediate reference in the literature.

\begin{lemma}[Comparison lemma]
\label{lem:compare}
Suppose that continuous function
$w$ and  $u$ are a viscosity 
sub and respectively super-solution of 
\eqref{eq:eikonaleq},
and satisfy the boundary conditions:
$$
w(0) = u(0) = w(2\pi) = u(2\pi) =0. 
$$
Then, $w \leq u$ on $ [0,2\pi]$.
\end{lemma}
\begin{proof}
Towards a counter-position, we assume that 
$\max_{[0, 2\pi]} [w- u] =: c_1 > 0.$
For small constants $\delta, \eps> 0$, set
$$
\pde(x, y)  := 
(1-\delta) w(x) - u(y) - \frac{1}{2 \eps} |x-y|^2 , \quad x,y \in [0, 2 \pi]. 
$$
Choose $\xsde, \ysde \in  [0,2\pi]$ satisfying
$\max_{[0, 2 \pi]} \pde = \pde(\xsde,\ysde)$.
Followings are elementary consequences.
\begin{enumerate}
\item 
\label{step:st1} 
Clearly, $0<c_1\le \pde(\xsde,\ysde) \le \|w\|_\infty+\|u\|_\infty$.
Also, 
$$
|\xsde-\ysde|^2\le 2 \eps ( \|w\|_\infty+\|u\|_\infty).
$$
\item 
\label{step:st2} 
Let $\xsd,\ysd$ be any limit point of $\xsde,\ysde$ as 
$\eps \downarrow 0$. 
By the above step, $\xsd=\ysd$. 
\item 
\label{step:st3} 
By definitions $(1-\delta)w(x)-u(x) = \pde(x,x)
\le \pde(\xsde,\ysde) \le w(\xsde)-u(\ysde)$
for any $x$. We use a limit argument to conclude
that  $(1-\delta)w(x)-u(x) \le w(\xsd)-u(\ysd)$.
\item 
\label{step:st4} 
Let $x^*$ be any limit point of $\xsd$ as $\delta \downarrow 0$. 
Then, for any $x$,
 \begin{align*} 
 w(x) - u(x)  &= \lim_{\delta \downarrow 0} \, 
 (1 - \delta)w(x) - u(x)  \leq  \liminf_{\delta  \downarrow 0}\,
(1-\delta)w(\xsd) - u(\xsd) \\
& = \hat{w}(x^*) - \tilde{w}(x^*). 
\end{align*}
Thus,  $w(x^*) -u(x^*) = \max_{[0, 2\pi]}[w-u]=: c_1 > 0$.
\item 
\label{step:st5} 
Since $w, u$ are continuous, 
 there exists $\epsilon_0$, $a \in(0,\pi/2)$ such that
 for every $\delta, \epsilon \in (0, \epsilon_0]$,
\begin{equation}
\label{eq:c0}
\xsde, \ysde  \in (a, 2\pi-a), \quad
\Rightarrow
\quad
 \sin^2(\xsde/2),   \sin^2(\ysde/2) \ge \sin^2(a/2).
\end{equation}
\end{enumerate}

We now proceed as in the usual comparison proof in the theory
of viscosity solutions which we provide for completeness.  We first observe that
$ x \in [0, 2 \pi] \mapsto \pde(\cdot, \ysde)$
is maximized at $\xsde$. Hence,
$$
\xsde \in \text{argmax}_{[0, 2\pi]}~ [w - \frac{1}{ (1-\delta)}\, \varphi],
\quad
\text{where}
\quad
\varphi(x):= \frac{1}{2 \epsilon}  |x -\ysde|^2.
$$
Since $w$ is a viscosity subsolution of \eqref{eq:eikonaleq}, 
the following inequality holds,
\begin{equation}
\label{eq:w}
\frac{1}{2(1-\delta)^2}\, |p_{\delta,\eps}|^2 \leq 2 \sin^2(\xsde/2),
\quad
\text{where}
\quad
p_{\delta,\eps} = \nabla \varphi(\xsde)= \frac{\xsde -\ysde}{\eps}.
\end{equation}

Proceeding almost exactly as above and using the fact
that $u$ is a viscosity supersolution, we arrive at the
following inequality,
\begin{equation}
\label{eq:u}
\frac{1}{2}\, |p_{\delta,\eps}|^2 \geq 2 \sin^2(\ysde/2) \ge 2 \sin^2(a/2)>0,
\end{equation}
where the final inequality follows from \eqref{eq:c0}.
We now subtract the above inequality from
\eqref{eq:w}. The result is the following,
$$ 
\frac{1}{2}(p_{\delta, \epsilon})^2 ( (1-\delta)^{-2} - 1 ) \leq 
2[\sin^2(\xsde/2)-\sin^2(\ysde/2)].
$$

We let $\epsilon \downarrow 0$ while keeping $\delta$ fixed. 
Then by Step~\ref{step:st2}, $|\xsde -\ysde|$ 
converges to zero.
Hence,
$$
\limsup_{\epsilon \downarrow 0}  \frac{1}{2}(p_{\delta, \epsilon})^2  \leq 0. 
$$
This is contradiction with \eqref{eq:u}.
\end{proof}


\bibliographystyle{abbrvnat}
\bibliography{kuramoto}

\end{document}